\documentclass[final]{siamltex}

\usepackage[dvips]{graphicx}
\usepackage{amsmath}
\usepackage{amssymb}
\usepackage{enumerate}
\usepackage{epsfig}
\usepackage{subfigure}

 \newcommand{\cF}{\mathcal{F}}

 \DeclareMathOperator{\cone}{cone}

 \DeclareMathOperator{\Int}{int}
 \newcommand{\ip}[2]{\ensuremath{\langle #1,#2\rangle}}

 \DeclareMathOperator{\mB}{\mathbb{B}}

    \newcommand{\mN}{\mathcal{N}}

 \newcommand{\mP}{\mathfrak{P}}

 \DeclareMathOperator{\Prob}{Prob}
 
 \newcommand{\Rset}{\mathbb{R}}
 
\newcommand{\w}{\omega}

\newtheorem{assumption}{Assumption}

 \title{Symmetric confidence regions and confidence intervals for normal map formulations of \\ stochastic variational inequalities}

\author{Shu Lu\thanks{Department of Statistics and Operations
Research, University of North Carolina at Chapel Hill, 355 Hanes
Building, CB\#3260, Chapel Hill, NC 27599-3260}
({\tt shulu@email.unc.edu}).}

\begin{document}

\maketitle

\begin{abstract}
Stochastic variational inequalities (SVI) model a large class of equilibrium problems subject to data uncertainty, and are closely related to stochastic optimization problems. The SVI solution is usually estimated by a solution to a sample average approximation (SAA) problem. This paper considers the normal map formulation of an SVI, and proposes a method
 to build asymptotically exact confidence regions and confidence intervals for the solution of the normal map formulation, based on the asymptotic distribution of SAA solutions. 
The confidence regions are single ellipsoids with high probability. We also discuss the computation of simultaneous and individual confidence intervals.
\end{abstract}


\section{Introduction}\label{s:intro}

This paper considers a stochastic variational inequality (SVI), in which the function that defines the variational inequality is an expectation function. Let $(\Omega, \cF, P)$ be a probability space, and $\xi$ be a random vector that is defined on $\Omega$ and supported on a closed subset $\Xi$ of $\Rset^d$. Let $O$ be an open subset of $\Rset^q$, and $F$ be a measurable function from $O\times \Xi$ to $\Rset^q$, such that for each $x\in O$ the expectation $f_0(x)=E[F(x,\xi)]$ is well defined.
Let $S$ be a polyhedral convex set in $\Rset^q$.
The SVI problem is to find a point $x\in S\cap O$ such that
\begin{equation}\label{eqn:svi}
 0\in f_0(x) + N_S(x),
\end{equation}
where $N_S(x)\subset \Rset^q$ denotes the normal cone to $S$ at $x$:
\[N_S(x)=\{v\in \Rset^q \mid \ip{v}{s-x}\le 0 \text{ for each } s\in S\}.\]
We use $\ip{\cdot}{\cdot}$ to denote the scalar product of two vectors of the same dimension. Here and elsewhere, the symbol $\subset$ stands for set inclusion and allows the two sets compared to coincide.

The function $f_0$ defined above is a function from $O$ to $\Rset^q$. We are interested in situations in which $f_0$ does not have a closed form expression and is approximated by a sample average function.
Let $\xi^1, \cdots, \xi^n$ be independent and identically distributed (i.i.d.) random variables with
distribution same as that of $\xi$. Define the sample average function $f_n: O\times \Omega \to \Rset^q$ by
\begin{equation}\label{eqn:def_fN}
f_n (x, \w)= n^{-1} \sum_{i=1}^n F(x,\xi^i(\w)).
\end{equation}
The sample average approximation (SAA) problem is to find a point $x\in S\cap O$ such that
\begin{equation}\label{eqn:saa}
0 \in f_n (x, \w) + N_S (x).
\end{equation}

In the rest of this paper, we will write $f_n(x,\w)$ as $f_n(x)$ when clear from the context. We will consider the \emph{normal map} formulations (to be defined below) of both \eqref{eqn:svi} and \eqref{eqn:saa}. The major objective is to develop a method to build confidence regions and confidence intervals for the solution of the normal map formulation of \eqref{eqn:svi}.
We will explain how to obtain confidence regions and intervals for the solution of \eqref{eqn:svi} in Section \ref{s:num}. For brevity, we refer to a solution to \eqref{eqn:svi} or its normal map formulation as a true solution, and a solution to \eqref{eqn:saa} or its normal map formulation as an SAA solution.

Under certain regularity conditions, the SAA solutions almost surely converge to a true solution as the sample size $n$ goes to infinity; see G\"{u}rkan, \"{O}zge and Robinson \cite{gur.ozg.smr:sps}, King and Rockafellar \cite{kin.rtr:ats}, and Shapiro, Dentcheva and Ruszczy\'{n}ski \cite[Section 5.2.1]{sha.den.rus:sp}. \cite[Theorem 2.7]{kin.rtr:ats} and \cite[Section 5.2.2]{sha.den.rus:sp} provided the asymptotic distribution of SAA solutions. Xu \cite{xu:saa} showed that the SAA solutions converge to the set of true solutions in probability at an exponential rate under some assumptions on the moment generating functions of certain random variables; see \cite{sha.xu:smp} for related results on the exponential convergence rate. For work on stability of stochastic optimization problems, see \cite{den:dst,dup.wets:abs,rom:ssp,sha:abo} and numerous references therein.

As discussed in Pflug \cite{pfl:sos}, there are two basic approaches to constructing confidence regions for the true solution of a stochastic optimization problem. The first approach is based on the asymptotic distribution of SAA solutions, and the second is based on properties about boundedness in probability with known tail behavior. The latter approach was used in \cite{pfl:sos} to construct universal confidence sets for the true solution, among other results. Vogel \cite{vog:ucs} continued that approach and developed a general method for universal confidence sets, taking account for random feasible sets and nonunique optimal solutions.

The method in this paper belongs to the first approach in the above classification. Our goal is to construct asymptotically exact confidence regions that are convenient to compute along with confidence intervals, under assumptions that guarantee the SVI and its SAA approximations to have locally unique solutions.
Our work is closely related to Demir \cite{dem:acr}. In that dissertation, Demir considered the normal map formulation of the SVI and obtained in \cite[Equation (36)]{dem:acr} an expression for confidence regions of the true solution. That expression depended on quantities not directly computable. \cite[Theorem 3.16]{dem:acr} modified it, resulting in a formula for the set called $S^n$ in \cite{dem:acr}. That formula looks similar to \eqref{eqn:conf_reg_nonsingular} of the present paper, but it uses $d(f_0)_S(z_n)(z_n-z)$ (in our notation) instead of $d(f_n)_S(z_n)(z-z_n)$ in \eqref{eqn:conf_reg_nonsingular}.
   We show in this paper that the set \eqref{eqn:conf_reg_nonsingular} is an asymptotically exact confidence region of the true solution, by proving that the probability for it to contain the true solution converges to the prescribed level of confidence.
 The development in \cite{dem:acr} did not justify its method with an asymptotic exactness result, aside from using some restrictive assumptions.


The starting point of the asymptotic analysis in this paper is Theorem \ref{t:asy_dis}, which is proved in \cite{lu.bud:crsvi} and related to results in \cite{dem:acr,kin.rtr:ats,sha.den.rus:sp}. 
  Equations \eqref{eqn:zN_dist_cov} and \eqref{eqn:zN_dist_cov2} in that theorem describe the asymptotic distribution of the solution (denoted by $z_n$) to the normal map formulation of the SAA problem \eqref{eqn:saa} in terms of a piecewise linear function $L_K$ and a normal random vector $Y_0$. Because $L_K$ depends on the true solution $z_0$ and is unknown before $z_0$ is found, we need to replace it by a suitable estimator in order to establish a confidence region for $z_0$. For the general situations considered in this paper, the dependence of $L_K$ on the location of $z_0$ is \emph{discontinuous}, due to the nonsmooth structure of $S$. Such discontinuity is a major issue to be addressed for establishing computable confidence regions, and this paper handles this issue using a different approach from those in \cite{lu:nmb,lu.bud:crsvi}.

In \cite{lu:nmb,lu.bud:crsvi}, $L_K$ is replaced by estimators designed to converge to $L_K$ in probability. The consistency of those estimators relies on the exponential convergence rate of $z_n$ to $z_0$. When $L_K$ is piecewise linear with multiple pieces, its estimators constructed in \cite{lu:nmb,lu.bud:crsvi} have multiple pieces with high probability, producing asymmetric confidence regions that are factions of ellipsoids pieced together. In the present paper, we use a property of general piecewise affine functions (Theorem \ref{t:dfx}) to show that for large $n$ the two vectors  $-d(f_n)_S(z_n)(z_0-z_n)$  and $L_K(z_n-z_0)$ are close to each other even though $d(f_n)_S(z_n)$ may be a very different function from $L_K$. Accordingly, we can use $-d(f_n)_S(z_n)(z_0-z_n)$ to replace $L_K(z_n-z_0)$ in \eqref{eqn:zN_dist_cov2} without losing the limiting property (Theorem \ref{t:d_f0_S_d_fN_S}). This leads to a computable confidence region for $z_0$ given in \eqref{eqn:conf_reg_nonsingular} or \eqref{eqn:R_n_epsilon} without the need of constructing a consistent estimator for $L_K$. The main advantage with the new method comes from Proposition \ref{p:invertible}, which implies that the confidence region built from this method is with high probability a single ellipsoid (or can be approximated by a degenerate ellipsoid in the singular case). This brings much efficiency for describing confidence regions or computing simultaneous confidence intervals. Another advantage of the new method is that it does not rely on the exponential convergence rate of $z_n$. Lastly, we mention that with this method the confidence region obtained from a particular $z_n$ may be very different from that obtained from another $z_n$, since $d(f_n)_S(z_n)$ can be very different  for different $z_n$. This does not conflict with the asymptotic exactness of the confidence region.

Given the results on confidence regions and simultaneous confidence intervals, a natural question is whether individual confidence intervals for components of $z_0$ can be computed in a parallel manner, by using $d(f_n)_S(z_n)$ to replace $L_K$ in \eqref{eqn:zN_dist_cov}. Since $d(f_n)_S(z_n)$ is an invertible linear function with high probability, the individual confidence interval given by this approach can be expressed by a closed-form formula \eqref{eqn:ind_ci_conv_formula}. Theorem \ref{t:ind_ci} shows that the probability for that interval to contain the $j$th component of $z_0$ converges to a quantity related to the random variable $\Gamma = (L_K)^{-1}(Y_0)$, and that quantity equals the desired confidence level $1-\alpha$ when the condition in \eqref{eqn:prob_intersection} holds. Based on that convergence result and given the simple format of \eqref{eqn:ind_ci_conv_formula}, one can use \eqref{eqn:ind_ci_conv_formula} as an approximative confidence interval for $(z_0)_j$ to compare with confidence intervals obtained from other methods (such as methods developed in \cite{lam.lu.bud:ici} to compute asymptotically exact individual confidence intervals using estimators in \cite{lu:nmb,lu.bud:crsvi}), which generally require more computation when $L_K$ has more than two pieces.


Examples of stochastic variational inequalities of the form \eqref{eqn:svi} include stochastic Nash equilibrium problems in which players' cost functions are expected values of certain random variables, such as the energy market problem studied in \cite{gur.ozg.smr:sps,hau.zac.leg.sme:sdn}.
Stochastic variational inequalities also arise as first-order conditions of optimization problems whose objective functions are expectations. If the exact values of some coefficients of the objective function are unknown, and are estimated using sample data average, then the solution obtained for such a problem is essentially an SAA solution. The method of this paper can be applied to such problems to provide a quantitative measure on the effect of
sample variations on solutions obtained for those problems. We have applied this method to a type of statistical learning problems \cite{lu.liu:cri}.

 Below we briefly introduce the normal map formulation of variational inequalities.
 The normal map induced by the function $f_0:O\to \Rset^q$ and the polyhedron $S\subset \Rset^q$ is defined to be a function $(f_0)_S: \Pi_S^{-1}(O)\to \Rset^q$, with
\begin{equation}\label{eqn:def_nm}
(f_0)_S(z)=f_0(\Pi_S(z))+ (z-\Pi_S(z))  \text{ for each }z\in \Pi_S^{-1}(O),
\end{equation}
 where $\Pi_S(z)$  denotes the Euclidean projection of $z$ on $S$, and $\Pi_S^{-1}(O)$ is the set of points $z\in\Rset^q$ such that $\Pi_S(z)\in O$. If a point $x\in S\cap O$ satisfies \eqref{eqn:svi}, then the point $z=x-f_0(x)$ satisfies $\Pi_S(z)=x$ and
\begin{equation}\label{eqn:vi_nm}
(f_0)_S(z) = 0.
\end{equation} Conversely, if $z$ satisfies
\eqref{eqn:vi_nm}, then $x=\Pi_{S}(z)$ satisfies $x-f_0(x)=z$ and solves \eqref{eqn:svi}. Thus, equation \eqref{eqn:vi_nm} is an equivalent formulation for \eqref{eqn:svi}, and is referred to as the normal map formulation of \eqref{eqn:svi}.

In general, for any function $g$ from (a subset of) $\Rset^q$ to $\Rset^q$ and any closed and convex set $C$ in $\Rset^q$, one can define the normal map induced by $g$ and $C$, denoted by $g_C$, in the same way as \eqref{eqn:def_nm} with $g$ in place of $f$ and $C$ in place of $S$. For example, the normal map induced by the sample average function $f_n$ and the set $S$ is a function $(f_n)_S:O\to \Rset^q$ defined as
\[
(f_n)_S(z)=f_n(\Pi_S(z))+ (z-\Pi_S(z)) \text{ for each } z\in \Pi_S^{-1}(O).
\]
The following is the normal map formulation for the SAA problem \eqref{eqn:saa}:
\begin{equation}\label{eqn:saa_z}
(f_n)_S(z)=0.
\end{equation}
Equation \eqref{eqn:saa_z} is related to \eqref{eqn:saa} in the same way as \eqref{eqn:vi_nm} is to \eqref{eqn:svi}.

The above definition for the normal map is from Robinson \cite{smr:nmi,smr:sav}. The normal map concept is closely related to the Minty parametrization \cite{min:mno}, in the sense that the variable $z$ can be considered as the parameter in the Minty parametrization of the graph of $N_S$: the mapping $z \to (\Pi_S(z), z-\Pi_S(z))$ is one-to-one from $\Rset^q$ onto the graph of $N_S$, and the equation \eqref{eqn:vi_nm} is a reformulation of \eqref{eqn:svi} in terms of $z$ using that parametrization.

Because $S$ is a polyhedral convex set by assumption, the Euclidean projector $\Pi_S$ is a piecewise affine function (see Section \ref{s:pa} below for the precise definition of piecewise affine functions). It coincides with an affine function on each of a family of finitely many $q$-dimensional polyhedral convex sets. This family is called the normal manifold of $S$, and each set in this family is called an $q$-cell, where the symbol $q$ refers to the dimension of those sets. The union of all the $q$-cells is $\Rset^q$. Any two distinct $q$-cells are either disjoint, or meet at a common proper face of them. (A face of a convex set $P$ in $\Rset^q$ is defined to be a convex
subset $F$ of $P$ such that if $x_1$ and $x_2$ belong to $P$ and
$\lambda x_1 + (1-\lambda) x_2 \in F$ for some $\lambda \in (0,1)$,
then $x_1$ and $x_2$ actually belong to $F$, see, e.g., \cite{rtr:ca}. $F$ is a proper face of $P$, if it is a nonempty face of $P$ and is not $P$ itself.) For detailed discussions on the normal manifold and properties of piecewise affine functions, see \cite{ral:npr,ral:bnn,smr:nmi,sch:ipd}.

To illustrate the normal manifold concept, consider the example when $S=\Rset^q_+$,  the nonnegative orthant in $\Rset^q$. For this example, the projector $\Pi_S$ coincides with an affine function when restricted to each fixed orthant of $\Rset^q$, and that affine function is different for a different orthant: for example we have $\Pi_S(z)=z$ for points $z\in \Rset^q_+$, $\Pi_S(z)=0$ for $z\in \Rset^q_-$, and $\Pi_S(z)=(z_1,0,\cdots,0)$ for $z\in \Rset_+ \times \Rset^{q-1}_-$, so $\Pi_S$ coincides with the identity function, the zero function, or the projector onto the $z_1$ axis, when restricted on $\Rset^q_+$, $\Rset^q_-$ or $\Rset_+ \times \Rset^{q-1}_-$ respectively. Accordingly, the normal manifold of $\Rset^q_+$ is the family of all orthants in $\Rset^q$.

Below we introduce some terminology and notation. 
A subset $K$ of $\Rset^q$ is called a cone if $\mu x\in K$ whenever $x\in K$ and $\mu$ is a positive real number. For a set $C\subset \Rset^q$, $\Int C$ denotes its interior, and $\cone C$ is the smallest cone that contains it. (This definition for $\cone C$ is equivalent to the one given in \cite[Section 2.1.1]{sch:ipd} for polyhedral convex sets $C$. In this paper we will only apply the definition to such sets.) We use $\|\cdot\|$ to denote the norm of an element in a normed space; unless explicitly stated otherwise, it can be any norm, as long as the same norm is used in all related contexts. We use $\mN(0, \Sigma)$ to denote a Normal random vector with covariance matrix $\Sigma$. Weak convergence of $k$-dimensional random variables $Y_n$ to $Y$ will be denoted as $Y_n \Rightarrow Y$, which means that $Ef(Y_n)$ converges to $Ef(Y)$ for all bounded continuous functional $f$ on $\Rset^k$. Following \cite{smr:ls3}, a locally Lipschitz function $g:\Rset^q\to\Rset^m$ is said to be \emph{B-differentiable} at a point $x_0\in \Rset^q$ if there is a positively homogeneous function $G: \Rset^q \to \Rset^m$, such that
\begin{equation}\label{eqn:def_Bdifferentiability}
g(x_0+v)=g(x_0)+ G(v)+ o(v).
\end{equation}
(Recall that a function $G$ is positively homogeneous, if $G(\lambda v) = \lambda G(v)$ for each nonnegative real number $\lambda$ and each $v\in \Rset^q$.) Such a function $G$ is called the B-derivative of $g$ at $x_0$ and is denoted as $dg(x_0)$. Note that $dg(x_0)(h)$ is exactly the directional derivative of $g$ at $x_0$ along the direction $h$. Indeed, as pointed out in \cite{sha:cdd}, for a locally Lipschitz function in finite-dimensional spaces, B-differentiability (called the directional differentiability in the sense of Fr\'{e}chet in \cite{sha:cdd}) of the function is equivalent to directional differentiability of the function for all directions. If $g$ is differentiable at $x_0$, then the B-derivative $dg(x_0)$ coincides with the standard Fr\'{e}chet derivative.

The rest of this paper is organized as follows. Section \ref{s:pa} is a discussion on general piecewise affine functions. Section \ref{s:asy} establishes the main asymptotic distribution results. Section \ref{s:sim_ci} provides methods to build confidence regions and simultaneous confidence intervals. Section \ref{s:ind_ci} discusses computation of individual confidence intervals. Section \ref{s:num} concludes the paper with numerical examples.

\section{Piecewise affine functions}\label{s:pa}

This section discusses general piecewise affine functions, and proves some properties to be used in subsequent development. The notation in this section is independent of the notation in the rest of this paper.

By definition, a continuous function
$f$ from $\Rset^q$ to $\Rset^m$ is called \emph{piecewise affine} if there exists a finite family of affine functions $f_j:
\Rset^q\to \Rset^m,j=1,\cdots,k$, such that the inclusion
$f(x)\in\{f_1(x),\cdots,f_k(x)\}$ holds for each $x\in \Rset^q$ \cite{sch:ipd}. The affine functions $f_j, \ j=1,\cdots,k$ are called selection functions of $f$. If all $f_j$'s are linear functions then $f$ is called \emph{piecewise linear}.

A closely related concept is the \emph{polyhedral subdivision} \cite{eav.rot:rpp,sch:ipd}. A polyhedral subdivision of $\Rset^q$ is a finite collection of polyhedral convex sets in $\Rset^q$, $\mP=\{P_1,\cdots,P_l\}$, that satisfies the following conditions:
\begin{enumerate}
 \item Each $P_i$ is a polyhedral convex set of dimension $q$.
 \item The union of all $P_i$ is $\Rset^q$.
 \item The intersection of each two $P_i$ and $P_j$, $1\le i\ne j\le l$, is either empty or a common proper face of both $P_i$ and $P_j$.
\end{enumerate}
If each $P_i\in \mP$ is a polyhedral convex cone, then $\mP$ is a \emph{conical subdivision}.
It was shown in \cite[Proposition 2.2.3]{sch:ipd} that for any piecewise affine function $f$ there corresponds a polyhedral subdivision $\mP$ of $\Rset^q$ such that $f$ coincides with an affine function on each $P\in \mP$. If $f$ is piecewise linear, then the corresponding $\mP$ is a conical subdivision. For example, if $S=\Rset^q_+$, then the Euclidean projector $\Pi_S$ is piecewise linear, and the family of all orthants in $\Rset^q$ is the corresponding  conical subdivision. In general, for any polyhedral convex set $S$ the Euclidean projector $\Pi_S$ is piecewise affine, with the normal manifold of $S$ being the corresponding polyhedral subdivision.

In the rest of this section, let $f: \Rset^q\to\Rset^m$ be a piecewise affine function with the corresponding subdivision $\mP$. Clearly, if $f$ is represented by different selection functions on $P_1\in \mP $ and $P_2\in \mP$, and $x\in P_1\cap P_2$, then $f$ is nondifferentiable at $x$. However, it is well known that $f$ is B-differentiable at any point in $\Rset^q$; below we explain a formula for the B-derivative of $f$ at a point $x\in \Rset^q$. For the derivation of this formula see \cite{eav.rot:rpp} and \cite[Proposition 2.2.6]{sch:ipd}. Let
\[
\mP(x)=\{P\in \mP\mid x\in P\}
\]
be the subfamily of $\mP$ that consists of elements in $\mP$ containing $x$. We use $|\mP(x)|$ to denote the union of all $P\in \mP(x)$, called  the \emph{underlying} set of $\mP(x)$ in algebraic topology \cite{mun:eat}. We caution the reader that $|\mP(x)|$ here does not mean the cardinality of $\mP(x)$. It is obvious that $x$ belongs to the interior of $|\mP(x)|$. For each $x\in \Rset^q$, define the following family of polyhedral convex cones:
\[
\mP'(x)=\{\cone(P-x) \mid P\in \mP(x)\}.
\]
The family $\mP'(x)$ is a conical subdivision of $\Rset^q$. The B-derivative $d f(x)$ of $f$ at $x$ is a piecewise linear function from $\Rset^q$ to $\Rset^m$, whose corresponding subdivision is exactly $\mP'(x)$. If $f$ coincides with the affine function $Ax+b$ on the polyhedral convex set $P\in \mP(x)$, then
\begin{equation}\label{eqn:dfx}
d f(x)(h)= A h \text{ for each } h\in \cone(P-x).
\end{equation}
A consequence of \eqref{eqn:dfx} is that the selection functions of $df(x)$ are exactly the linear parts of selection functions of $f$ on elements of $\mP(x)$. In particular, if $x$ is contained in the interior of some $P\in \mP$, then $df(x)$ is a linear map.

While for a fixed point $x\in \Rset^q$ the B-derivative $df(x)$ is a continuous function on $\Rset^q$, the dependence of $df(x)$ on $x$ is discontinuous, because $df(x)$ changes abruptly to a very different function as $x$ moves from the interior of some $P\in \mP$ to its boundary. For example consider $\Pi_S$ with $S=\Rset^q_+$ again. As a piecewise linear function, $\Pi_S$ is B-differentiable at any point $z\in \Rset^q$. At any $z$ in the interior of $\Rset^q_+$, the B-derivative $d\Pi_S(z)$ is the identity map. At any $z$ with $z_1=0$ and $z_i>0$ for $i=2,\cdots,q$, $d\Pi_S(z)$ is a piecewise linear map with two pieces:
\[
\text{For each }h\in \Rset^q, \
d \Pi_S(z)(h)= \left\{
\begin{array}{ll}
h, & \text{ if } h_1\ge 0,\\
(0,h_2,\cdots,h_q), & \text{ if } h_1 \le 0.
\end{array}
\right.
\]

As will become clear, the discontinuity of $df(x)$ with respect to $x$ is a major issue to be addressed to develop methods for confidence regions in this paper and in \cite{lu:nmb,lu.bud:crsvi}. In this paper, we handle this issue by utilizing a ``symmetry'' property between the B-derivatives $df(x)$ and $df(y)$ for two points $x$ and $y$ that belong to a common set in $\mP$, shown in Theorem \ref{t:dfx} below. Theorem \ref{t:dfx} will be used to establish our main result in Theorem \ref{t:d_f0_S_d_fN_S}. The proof of Theorem \ref{t:dfx} uses the following lemma, which gives two equivalent statements for the condition $y\in |\mP(x)|$.

\begin{lemma}\label{l:mPx_mPy}
Let $x$ and $y$ be two points in $\Rset^q$. The following are equivalent.
\begin{enumerate}[(1)]
\item $y\in |\mP(x)|$.
\item $x\in |\mP(y)|$.
\item There exists $P\in \mP$ such that both $x$ and $y$ belong to $P$.
\end{enumerate}
\end{lemma}

\begin{proof}
Suppose $y\in |\mP(x)|$. Then there exists $P\in \mP(x)$ such that $y\in P$. The fact that $P\in \mP(x)$ implies $x\in P$. This proves the direction (1)$\Rightarrow$(3).

Now suppose (3) holds. Then $P\in \mP(x)$. Since $y\in P$, we have $y\in |\mP(x)|$. This proves (3)$\Rightarrow$(1). It follows that (3) and (1) are equivalent. Similarly we can prove (3) and (2) are equivalent. The equivalence between (1) and (2) follows.
\end{proof}

\begin{theorem}\label{t:dfx}
Let $x\in \Rset^q$, and let $y\in |\mP(x)|$. Then
\[
d f (x) (y-x) = - df(y)(x-y).
\]
\end{theorem}

\begin{proof}
By Lemma \ref{l:mPx_mPy}, there exists $P\in \mP$ that contains both $x$ and $y$. Let $A$ be the matrix for the linear part of the selection function of $f$ on $P$. The fact that $y\in P$ implies $y-x\in \cone (P-x)$, so by \eqref{eqn:dfx} we have $df(x)(y-x)=A(y-x)$. The fact $x\in P$ implies $x-y\in \cone(P-y)$. Again by \eqref{eqn:dfx} we have $df(y)(x-y)=A(x-y)$.
\end{proof}
%

Proposition \ref{p:mMy_subset_mMx} below shows that $\mP(x)$ is a superset of $\mP(y)$ for all $y$ sufficiently close to $x$.
Recall from comments below \eqref{eqn:dfx} that the selection functions of $df(x)$ are exactly the linear parts of selection functions of $f$ on elements of $\mP(x)$. Therefore, a consequence of Proposition \ref{p:mMy_subset_mMx} is that
the family of selection functions of $df(x)$ contains the family of selection functions of $df(y)$,
for all $y$ sufficiently close to $x$.
\begin{proposition}\label{p:mMy_subset_mMx}
Let $x\in \Rset^q$. There exists a neighborhood $X$ of $x$, such that each $y\in X$ satisfies $\mP(y) \subset \mP(x)$.
\end{proposition}

\begin{proof}
Recall that $x$ belongs to the interior of $|\mP(x)|$. Let $X$ be a neighborhood of $x$ in the interior of $|\mP(x)|$, and let $y\in X$. Since $y$ belongs to the interior of $|\mP(x)|$, there exist finitely many polyhedrons $P_1,\cdots, P_k$ in $\mP(x)$, such that $y\in P_i$ for each $i=1,\cdots,k$ and $y$ belongs to the interior of $\cup_{i=1}^k P_i$. Now, let $P\in\mP(y)$; we shall prove that $P=P_i$ for some $i=1,\cdots,k$, by showing that it is impossible for a set $P\in \mP$ that is different from any of those $P_i$'s to meet $\cup_{i=1}^k P_i$. Suppose for the purpose of contradiction that $P$ is different from any $P_i$, $i=1,\cdots,k$. Then, for each $i$ the intersection $P\cap P_i$ is a proper face of $P$. Consequently, the intersection between $P$ and $\cup_{i=1}^k P_i$ is the union of finitely many polyhedrons of dimensions less than $q$. On the other hand, since $P$ is of dimension $q$ and $\cup_{i=1}^k P_i$ contains $y$ in its interior, the intersection between $P$ and $\cup_{i=1}^k P_i$ contains a full-dimensional convex set. This leads to a contradiction. We have thereby proved that $P=P_i$ for some $i=1,\cdots,k$. It follows that $\mP(y) \subset \mP(x)$.
\end{proof}

\section{Limiting properties}\label{s:asy}

This section provides some limiting properties of solutions to \eqref{eqn:saa_z}. Those properties will be used to develop methods on confidence region and simultaneous confidence intervals.

We make two sets of assumptions. Assumption \ref{assu1} below is to obtain nice properties of $f_0$ and $f_n$ about integrability and convergence. Following that, Assumption \ref{assu2} is to guarantee the existence, local uniqueness and stability of the true solution. The notation $T_S(x)$ that appears in Assumption \ref{assu2} denotes the tangent cone to $S$ at a point $x\in S$. Since $S$ is a polyhedral convex set, the following definition applies:
\begin{equation}\label{eqn:def_TSx}
T_S(x)= \{v\in \Rset^q \mid \text{ there exists } t>0  \text{ such that } x + tv \in S\}.
\end{equation}

\begin{assumption}\label{assu1}
(a)   $E\|F( x,\xi)\|^2 < \infty$ for all $x \in O$.\\
\noindent (b) The map $x \mapsto F(x, \xi(\omega))$ is continuously differentiable on $O$ for a.e. $\omega \in \Omega$, and
$E\|d_xF( x,\xi)\|^2 < \infty$ for all $x \in O$.\\
\noindent (c) There exists a square integrable random variable $C$ such that
$$\|F(x, \xi(\omega)) - F(x', \xi(\omega))\| + \|d_xF(x, \xi(\omega)) - d_xF(x', \xi(\omega))\| \le C(\omega) \|x-x'\|,$$
for all $x,x'\in O$ and a.e. $\omega \in \Omega$.
\end{assumption}

The notation $d_xF(x, \xi(\omega))$ here stands for the partial derivative of $F$ w.r.t. $x$, a $q\times q$ matrix.
The norms used in Assumption \ref{assu1} can be any norms in the $\Rset^q$ or $\Rset^{q\times q}$ space, since all norms in a finite-dimensional space are equivalent to each other.
A consequence of Assumption \ref{assu1} is the continuous differentiability of $f_0$ on $O$. Moreover, for any nonempty compact subset $X$ of $O$, let $C^{1} (X, \Rset^q)$ be the Banach space of continuously differentiable mappings $f:X \to \Rset^q$, equipped with the norm
\begin{equation}\label{eqn:norm_C1}
\|f\|_{1,X} = \sup_{x\in X} \|f(x)\| + \sup_{x\in X} \|d f(x)\|.
\end{equation}
Under Assumption \ref{assu1}, the sample average function $f_n$ converges to $f_0$ almost surely as an element of $C^1(X,\Rset^q)$, see, e.g., \cite[Theorems 7.44, 7.48 and 7.52]{sha.den.rus:sp} and \cite[Theorem 3]{lu.bud:crsvi}.

\begin{assumption}
	\label{assu2}
	Suppose that $x_0$ solves the variational inequality \eqref{eqn:svi}. Let $z_0=x_0-f_0(x_0)$, $L=d f_0(x_0)$, $K=T_S(x_0) \cap \{z_0-x_0\}^\perp$, and assume that the normal map  $L_K$ induced by $L$ and $K$, defined as $L_K(h)=L(\Pi_K(h))+h-\Pi_K(h)$ for each $h\in \Rset^q$, is a homeomorphism from $\Rset^q$ to $\Rset^q$.
	\end{assumption}

We assume in the above assumption that \eqref{eqn:svi} has a solution $x_0$. To guarantee the existence of such a solution, one would need to put additional conditions on $f_0$ and $S$. For example, if $f_0$ is strongly monotone on $S$ then \eqref{eqn:svi} must have a (globally unique) solution. Detailed discussions and more general solution existence conditions for variational inequalities can be found in \cite[Chapter 2]{fac.pan:fdv}, \cite[Chapter 12]{rtr.wet:va} and references therein.
The set $K$ defined in Assumption \ref{assu2} is called the critical cone to $S$ associated with $z_0$. The homeomorphism condition on $L_K$ guarantees that $x_0$ is a locally unique solution of \eqref{eqn:svi}, and that \eqref{eqn:svi} continues to have a locally unique solution around $x_0$ under small perturbation of $f_0$, see \cite[Lemma 1]{lu.bud:crsvi} and the original result in \cite{smr:sav}. Being the normal map induced by a linear map and a polyhedral convex cone, $L_K$ is a piecewise linear function. It was shown in \cite{smr:nmi} that $L_K$ is a homeomorphism if and only if it is coherently oriented (a piecewise linear function is coherently oriented if the determinants of matrices representing its selection functions all have the same nonzero sign), see also \cite{ral:bnn,sch:pbn} for shortened proofs.
A special case in which  the coherent orientation condition holds is when the restriction of $L$ on the linear span of $K$
is positive definite. In particular, if $f_0$ is strongly monotone on $O$, then the entire matrix $L$ is positive definite and $L_K$ is a global homeomorphism.

  The normal map $L_K$ is related to the normal map $(f_0)_S$ in \eqref{eqn:def_nm} in the following way.  Because $f_0$ is differentiable at $x_0$ and $\Pi_S$ is B-differentiable, the normal map $(f_0)_S$ is B-differentiable at $z_0$ by the chain rule of B-differentiability, with
\begin{equation}\label{eqn:d_f0_S_h}
d (f_0)_S (z_0) (h) = d f_0 (x_0) (d \Pi_S(z_0)(h)) + h - d \Pi_S(z_0)(h).
\end{equation}
 It was shown in \cite{smr:sav} that $L_K$ is exactly $d (f_0)_S (z_0)$, the B-derivative of $(f_0)_S$ at $z_0$.

The following theorem is adapted from \cite[Theorem 7]{lu.bud:crsvi}.
Here and hereafter, we use $\Sigma_0$ to denote the covariance matrix of $F(x_0,\xi)$.
 Equation \eqref{eqn:zN_dist_cov} in the theorem looks similar to the equation in \cite[Theorem 2.7]{kin.rtr:ats} and  \cite[Equation (5.74)]{sha.den.rus:sp}. Those two references are about the asymptotic distribution of $x_n$, a solution to \eqref{eqn:saa}, while \eqref{eqn:zN_dist_cov} here describes the asymptotic distribution of $z_n$, a solution to \eqref{eqn:saa_z}.

\begin{theorem}\label{t:asy_dis}
Suppose that Assumptions \ref{assu1} and \ref{assu2} hold. Let $Y_0$ be a normal random vector in $\Rset^q$ with zero mean and covariance matrix $\Sigma_0$. Then there exist neighborhoods $X_0$ of $x_0$ in $O$ and $Z$ of $z_0$ in $\Rset^q$ such that the following hold. For almost every $\w\in \Omega$, there exists an integer $N_{\w}$, such that for each $n \ge N_{\w}$, the equation
\eqref{eqn:saa_z}
 has a unique solution $z_n$ in $Z$, and the variational inequality \eqref{eqn:saa} has a unique solution in $X_0$ given by $x_n=\Pi_S(z_n)$. Moreover, $\lim_{n\to\infty}z_n = z_0$ and $\lim_{n\to\infty}x_n = x_0$ almost surely,
\begin{equation}\label{eqn:zN_dist_cov}
\sqrt{n}(z_n - z_0) \Rightarrow (L_K)^{-1}(Y_0),
\end{equation}
and
 \begin{equation}\label{eqn:zN_dist_cov2}
 \sqrt{n} L_K (z_n- z_0) \Rightarrow Y_0.
 \end{equation}
\end{theorem}

From \eqref{eqn:zN_dist_cov2}, the random variable $n [L_K(z_n-z_0)]^T \Sigma_0^{-1} [L_K(z_n-z_0)]$ weakly converges to a $\chi^2$ random variable with $q$ degrees of freedom (assuming $\Sigma_0$ is nonsingular). This leads to an expression for confidence regions of $z_0$. However, that expression includes the normal map $L_K$ in it, which is unknown unless $z_0$ is known because both $L$ and $K$ depend on $z_0$ or $x_0$. Therefore, to obtain a computable confidence region, one needs to substitute $L_K$ by a function that depends only on $z_n$ and $x_n$, without losing the limiting property.

Recall that $L_K$ is exactly $d (f_0)_S (z_0)$, the B-derivative of the normal map $(f_0)_S$ at $z_0$. A natural estimator of $L_K$ is therefore $d (f_n)_S(z_n)$. However, from the discontinuity of B-derivatives of piecewise affine functions discussed above Lemma \ref{l:mPx_mPy} and given that $\Pi_S$ is a piecewise affine function, the B-derivative $d \Pi_S(z_n)$ is not guaranteed to converge to $d \Pi_S(z_0)$ even though $z_n$ converges to $z_0$ almost surely. From \eqref{eqn:d_f0_S_h} it is clear that $d(f_n)_S(z_n)$ is not guaranteed to converge to $L_K$. To provide an alternative estimator of $L_K$, the papers \cite{lu:nmb} and \cite{lu.bud:crsvi} designed some functions to converge to $L_K$, by utilizing the exponential convergence rate of $z_n$ to $z_0$ in probability obtained under additional assumptions.

The approach taken here is different from methods in \cite{lu:nmb,lu.bud:crsvi}. Instead of designing functions that converge to $L_K$, we directly use $d(f_n)_S(z_n)$ to replace $L_K$ in \eqref{eqn:zN_dist_cov2}, knowing that $d(f_n)_S(z_n)$ is not guaranteed to converge to $L_K$ and may be very different with different $z_n$. The method is based on Theorem \ref{t:d_f0_S_d_fN_S}, which says that for large $n$ the vector $d(f_n)_S(z_n)(z_0-z_n)$ is close to $-L_K(z_n-z_0)$ with high probability, and that using $d(f_n)_S(z_n)$ to replace $L_K$ in \eqref{eqn:zN_dist_cov2} along with some sign changes keeps the weak convergence result to hold. As mentioned earlier, the proof of Theorem \ref{t:d_f0_S_d_fN_S} relies on a  ``symmetry'' property between the B-derivatives at two different points given in Theorem \ref{t:dfx}. From Theorems \ref{t:d_f0_S_d_fN_S} we derive Theorems \ref{t:zN_weakconv_nonsinuglar} and \ref{t:zN_weakconv_sinuglar}, to show that the set \eqref{eqn:conf_reg_nonsingular}, or \eqref{eqn:R_n_epsilon} for the singular case, is an asymptotically exact confidence region for $z_0$: that is, the set contains $z_0$ with probability converging to the prescribed level of confidence.

Comparing to methods in \cite{lu:nmb,lu.bud:crsvi}, the main advantage with the new method is given by Proposition \ref{p:invertible}, which states that $d(f_n)_S(z_n)$ is with high probability an invertible linear function. This implies that the set \eqref{eqn:conf_reg_nonsingular} is with high probability a single ellipsoid, when $\Sigma_0$ is nonsingular. When $\Sigma_0$ is singular one could use  $R_{n,0}$ in \eqref{eqn:R_n_0} to approximate $R_{n,\epsilon}$ in \eqref{eqn:R_n_epsilon}, and $R_{n,0}$ is with high probability a degenerate ellipsoid. In contrast, when $L_K$ is piecewise linear with multiple pieces, the estimators designed in \cite{lu:nmb,lu.bud:crsvi} have multiple pieces with high probability, providing confidence regions that are factions of ellipsoids pieced together. Clearly, it is much easier to describe a single ellipsoid, or to find the minimal enclosing box of a single ellipsoid to obtain simultaneous confidence intervals. Additionally, the new method does not rely on the assumptions for the exponential convergence rate of $z_n$, another advantage comparing to \cite{lu:nmb,lu.bud:crsvi}.



Before proceeding to the proofs, recall that the Euclidean projector $\Pi_S$ coincides with an affine function on each $q$-cell in the normal manifold of $S$. We use $\mP$ to denote the normal manifold of $S$, which is the polyhedral subdivision of $\Rset^q$ corresponding to $\Pi_S$. As in Section \ref{s:pa} we use $\mP(z_0)$ to denote the family of $q$-cells containing $z_0$. Lemma \ref{l:d_Pi_S} below will be used in the proof of Theorem \ref{t:d_f0_S_d_fN_S}.

\begin{lemma}\label{l:d_Pi_S}
Suppose that Assumptions \ref{assu1} and \ref{assu2} hold. Then for almost every $\w\in \Omega$ there exists an integer $N_\w$ such that the following equality holds for each $n\ge N_\w$:
 \begin{equation}\label{eqn:d_Pi_S}
d\Pi_S(z_0)(z_n-z_0)+ d\Pi_S(z_n)(z_0-z_n) = 0.
\end{equation}
\end{lemma}

\begin{proof}
Recall that $z_0$ belongs to the interior of $|\mP(z_0)|$. Since $z_n$ converges to $z_0$ w.p. 1, for almost every $\w\in \Omega$ there exists an integer $N_\w$ such that $z_n$ belongs to $|\mP(z_0)|$ for each $n\ge N_\w$. It follows from Theorem \ref{t:dfx} that each such $z_n$ satisfies
\eqref{eqn:d_Pi_S}.
\end{proof}

Theorem \ref{t:d_f0_S_d_fN_S} below is the main result of this paper.

\begin{theorem}\label{t:d_f0_S_d_fN_S}
Suppose that Assumptions \ref{assu1} and \ref{assu2} hold. Then for each $\epsilon>0$ we have
\begin{equation}\label{eqn:d_f0_d_fN}
\lim_{n\to\infty} \Prob\{\sqrt{n}\|d(f_0)_S(z_0)(z_n-z_0)+ d(f_n)_S(z_n)(z_0-z_n)\|> \epsilon\} = 0.
\end{equation}
Consequently,
\begin{equation}\label{eqn:zN_dist_cov_dfN}
 -\sqrt{n} d(f_n)_S(z_n) (z_0- z_n) \Rightarrow Y_0.
 \end{equation}
\end{theorem}

\begin{proof}
Recall from Assumption \ref{assu2} and Theorem \ref{t:asy_dis} that $x_0=\Pi_S(z_0)$ and $x_n=\Pi_S(z_n)$ are solutions to \eqref{eqn:svi} and \eqref{eqn:saa} respectively. From \eqref{eqn:d_f0_S_h} we have
\begin{equation}\label{eqn:d_f0_S}
d (f_0)_S (z_0) (z_n-z_0) = d f_0 (x_0) (d \Pi_S(z_0)(z_n-z_0)) + z_n-z_0 - d \Pi_S(z_0)(z_n-z_0).
\end{equation}
Similarly,
\begin{equation}\label{eqn:d_fN_S}
d (f_n)_S (z_n) (z_0-z_n)  = d f_n (x_n) (d \Pi_S(z_n)(z_0-z_n)) + z_0-z_n - d \Pi_S(z_n)(z_0-z_n).
\end{equation}
It follows that $\sqrt{n}\|d f_0)_S (z_0) (z_n-z_0) + d (f_n)_S (z_n) (z_0-z_n) \|$ is bounded from above by the sum of the following two terms.
\begin{description}
\item Term (a): $\sqrt{n}\|d f_0 (x_0) (d \Pi_S(z_0)(z_n-z_0)) +d f_n (x_n) (d \Pi_S(z_n)(z_0-z_n))\|$.
\item Term (b): $\sqrt{n}\|d \Pi_S(z_0)(z_n-z_0)+ d \Pi_S(z_n)(z_0-z_n)\|$.
\end{description}
By Lemma \ref{l:d_Pi_S}, term (b) converges to 0 almost surely, so it converges to 0 in probability. It remains to show that Term (a) converges to 0 in probability. Term (a) is bounded from above by the sum of the following two terms.
\begin{description}
\item Term (c): $\sqrt{n}\|d f_0 (x_0) (d \Pi_S(z_0)(z_n-z_0)) -d f_n (x_n) (d \Pi_S(z_0)(z_n-z_0))\|$.
\item Term (d): $\sqrt{n}\|d f_n (x_n) (d \Pi_S(z_0)(z_n-z_0)) +d f_n (x_n) (d \Pi_S(z_n)(z_0-z_n))\|$.
\end{description}
Because $df_n(x_n)(\cdot)$ is a linear map, by Lemma \ref{l:d_Pi_S} term (d) converges to 0 almost surely and therefore converges to 0 in probability. To prove the theorem it suffices to show that term (c) converges to 0 in probability.

Term (c) is bounded from above by the following
\begin{equation}\label{eqn:bdd}
\|d f_0 (x_0)- d f_n (x_n)\| \ \|\sqrt{n}(d \Pi_S(z_0)(z_n-z_0))\|
\end{equation}
where $\|d f_0 (x_0)- d f_n (x_n)\|$ is the norm of the linear operator $d f_0 (x_0)- d f_n (x_n)$, which is bounded from above by
\[
\|d f_0 (x_0)- d f_0 (x_n)\| + \|d f_0 (x_n)- d f_n (x_n)\|.
\]
Since $x_n$ converges to $x_0$ almost surely and $f_0$ is continuously differentiable under Assumption \ref{assu1}, $\|d f_0 (x_0)- d f_0 (x_n)\|$ converges to 0 almost surely. Assumption \ref{assu1} also guarantees $f_n$ to converge to $f_0$ almost surely as an element of $C^1(X,\Rset^q)$ for any compact subset $X$ of $O$. Let $X$ be a compact subset of $O$ that contains $x_0$ in its interior; we have $x_n \in X$ for all sufficiently large $n$. It follows that $\|d f_0 (x_n)- d f_n (x_n)\|$ converges to zero almost surely. This proves that $\|d f_0 (x_0)- d f_n (x_n)\|$ converges to zero almost surely. On the other hand, since $d \Pi_S(z_0)$ is positively homogenous, we can rewrite $\sqrt{n}(d \Pi_S(z_0)(z_n-z_0))$ as $d \Pi_S(z_0)(\sqrt{n}(z_n-z_0))$. By \eqref{eqn:zN_dist_cov}, $\sqrt{n}(z_n-z_0)$ converges in distribution to a random variable, so it is (uniformly) tight; see, e.g., \cite[Theorem 2.4]{vaa:ass}. It follows that $d \Pi_S(z_0)(\sqrt{n}(z_n-z_0))$ is uniformly tight, and that
the quantity \eqref{eqn:bdd} converges to 0 in probability. This proves that term (c) converges to 0 in probability, and thereby proves \eqref{eqn:d_f0_d_fN}.
Equation \eqref{eqn:zN_dist_cov_dfN} is a result of \eqref{eqn:zN_dist_cov2} and \eqref{eqn:d_f0_d_fN}.
\end{proof}

 Proposition \ref{p:invertible} below shows the function $d(f_n)_S(z_n)$ that appears in \eqref{eqn:zN_dist_cov_dfN} is a linear invertible function with high probability. Its proof uses the following lemma.

\begin{lemma}\label{l:local_homeomorphism}
Under Assumptions \ref{assu1} and \ref{assu2}, there exist a neighborhood $Z'$ of $z_0$ in $\Rset^q$ and a neighborhood $\mathbb{L}$ of the matrix $L$ in $\Rset^{q\times q}$, such that for each $z'\in Z'$ and $L'\in \mathbb{L}$ the map $\Upsilon(z',L'): \Rset^q\to \Rset^q$ defined as
\begin{equation}\label{eqn:def_Upsilon}
\Upsilon(z',L')(h)=L' d \Pi_S(z') (h)+ h - d \Pi_S(z')(h) \text{ for each } h\in \Rset^q
\end{equation}
is a global homeomorphism from $\Rset^q$ to $\Rset^q$.
\end{lemma}

\begin{proof}
First, note that for each $z'\in \Rset^q$ the B-derivative $d\Pi_S(z')$ is exactly the Euclidean projector onto the critical cone to $S$ associated with $z'$ \cite{pan:nmb,smr:ift}. Hence, the map $\Upsilon(z',L')$ is exactly the normal map induced by $L'$ and the latter critical cone. Recall that such a normal map is a global homeomorphism if and only if it is coherently oriented (see the discussion below Assumption \ref{assu2}). In the rest of this proof, we find neighborhoods $Z'$ and $\mathbb{L}$ to guarantee $\Upsilon(z',L')$ to be coherently oriented for $z'\in Z'$ and $L'\in \mathbb{L}$.

Since $K$ is the critical cone to $S$ associated with $z_0$, we have $d\Pi_S(z_0)=\Pi_K$ and $\Upsilon(z_0,L)=L_K$. Because $L_K$ is a global homeomorphism by assumption, it is coherently oriented. The fact that $K$ is a polyhedral convex cone implies that $L_K$ is a piecewise linear function. Let $M_1, \cdots, M_k$ be the matrices that represent the selection functions of $L_K$.  Since $L_K$ is coherently oriented, the determinants of $M_1, \cdots, M_k$ have the same nonzero sign (see the definition of the coherent orientation condition below Assumption \ref{assu2}). By choosing $ \mathbb{L}$ to be a sufficiently small neighborhood of $L$, we can guarantee that the determinants of matrices representing selection functions of $\Upsilon(z_0,L')$ to have the same nonzero sign for each $L'\in \mathbb{L}$. This proves that $\Upsilon(z_0,L')$ is coherently oriented for each $L'\in \mathbb{L}$.

By Proposition \ref{p:mMy_subset_mMx}, there exists a neighborhood $Z'$ of $z_0$ such that $\mP(z')\subset \mP(z_0)$ for each $z'\in Z'$. As noted in the remark above Proposition \ref{p:mMy_subset_mMx}, the family of the selection functions of $d\Pi_S(z_0)$ includes that of $d\Pi_S(z')$ for each $z'\in Z'$. Thus, for each $L'\in \mathbb{L}$ and $z'\in Z'$, any selection function of $\Upsilon(z',L')$ is also a selection function of $\Upsilon(z_0,L')$, and the coherent orientation of $\Upsilon(z',L')$ follows from the coherent orientation of $\Upsilon(z_0,L')$.
\end{proof}


\begin{proposition}\label{p:invertible}
Under Assumptions \ref{assu1} and \ref{assu2},
\[
\lim_{n\to\infty} \Prob\{d(f_n)_S(z_n) \text{ is an invertible linear map}\} =1.
\]
\end{proposition}

\begin{proof}
By the chain rule,
\[
d (f_n)_S (z_n) (h)  = d f_n (x_n) (d \Pi_S(z_n)(h)) + h - d \Pi_S(z_n)(h) \text{ for each } h\in \Rset^q.
\]
The linearity of $d (f_n)_S (z_n)$ depends on the linearity of $d\Pi_S(z_n)$. The latter function is linear whenever $z_n$ belongs to the interior of an $q$-cell in the normal manifold of $S$. Let $R$ be the union of boundaries of all the $q$-cells; it follows that $R$ is the union of finitely many polyhedral convex sets of dimensions less than $q$. Consider the tangent cone $T_R(z_0)$, which is defined in the same way as $T_S(x)$ in \eqref{eqn:def_TSx} if $z_0\in R$ and is the empty set if $z_0\not\in R$.
Let $Z_0$ be a neighborhood of $z_0$ such that
$Z_0 \cap R = Z_0 \cap (z_0+T_R(z_0))$; we have
\begin{equation}\label{eqn:TR_z0}
\sqrt{n}(z-z_0) \in T_R(z_0) \text{ for each } z\in Z_0\cap R \text{ and each integer }n.
\end{equation}
The fact that $z_n$ converges to $z_0$ almost surely implies
\[
\lim_{n\to\infty} \Prob \{z_n\not\in Z_0\} = 0.
\]
On the other hand, \eqref{eqn:TR_z0} implies
\[
\begin{split}
\Prob \{z_n\in Z_0\cap R\} \le \Prob \{\sqrt{n}(z_n-z_0)\in T_R(z_0)\}.
\end{split}
\]
It follows from \eqref{eqn:zN_dist_cov} and the fact that $L_K$ is a piecewise linear homeomorphism that
\[
\begin{split}
\lim_{n\to\infty} \Prob \{\sqrt{n}(z_n-z_0)\in T_R(z_0)\}\le \Prob \{(L_K)^{-1}(Y_0) \in T_R(z_0)\}=0.
\end{split}
\]
Because $
\Prob \{z_n\in R\} \le \Prob \{z_n\not\in Z_0\} + \Prob \{z_n\in Z_0\cap R\}$, we have proved
\begin{equation}\label{eqn:prob_R}
\lim_{n\to\infty} \Prob \{z_n\in R\} = 0,
\end{equation}
which implies that the probability for $d(f_n)_S(z_n)$ to be a linear function converges to 1 as $n$ goes to $\infty$.

Next, choose neighborhoods $Z'$ of $z_0$ and $\mathbb{L}$ of $L$ as in Lemma \ref{l:local_homeomorphism}. Since $z_n$ converges to $z_0$ almost surely, $z_n$ belongs to $Z'$ almost surely for sufficiently large $n$. It was shown in the proof of Theorem \ref{t:d_f0_S_d_fN_S} that $d f_n(x_n)$ almost surely converges to $d f_0(x_0)$. Consequently, $d f_n(x_n)$ belongs to $\mathbb{L}$ almost surely for sufficiently large $n$. By Lemma \ref{l:local_homeomorphism}, $d(f_n)_S(z_n)$ is a global homeomorphism almost surely for sufficiently large $n$, so the probability for $d(f_n)_S(z_n)$ to be an invertible function converges to 1 as $n$ goes to $\infty$. The conclusion of the proposition follows by combining the linearity and invertibility.
\end{proof}

The random variable $Y_0$ in \eqref{eqn:zN_dist_cov_dfN} has covariance matrix $\Sigma_0$, which depends on the true solution $x_0$. In computing confidence regions and intervals we will replace $\Sigma_0$ by $\Sigma_n$, the sample covariance matrix of $\{F(x_n, \xi^i)\}_{i=1}^n$. The lemma below supports such a replacement. 

\begin{lemma}\label{l:Sigma_n_Sigma_0}
Suppose that Assumptions \ref{assu1} and \ref{assu2} hold. The matrix $\Sigma_n$ converges to $\Sigma_0$ almost surely as $n \to\infty$.
\end{lemma}

\begin{proof}
Under the assumptions, $\frac{1}{n} \sum_{i=1}^n F(x,\xi^i(\w)) F(x,\xi^i(\w))^T$ converges to $E[F(x,\xi) F(x,\xi)^T]$ almost surely uniformly on each compact subset $X$ of $O$ (see, e.g., \cite[Theorem 7.48]{sha.den.rus:sp}). Since $x_n$ converges to $x_0$ almost surely, we have
\[\lim_{n\to\infty} \frac{1}{n} \sum_{i=1}^n F(x_n,\xi^i(\w)) F(x_n,\xi^i(\w))^T = E[F(x_0,\xi) F(x_0,\xi)^T] \text{ almost surely}.\]
Similarly $f_n(x_n)$ converges to $f_0(x_0)$ almost surely as $n \to \infty$.
This proves the lemma.
\end{proof}

%
%
%

%

\section{Confidence regions and simultaneous confidence intervals}\label{s:sim_ci}

This section provides formulas for confidence regions of $z_0$ that are computable from $z_n$, and provides a method to compute simultaneous confidence intervals for components of $z_0$. We treat two cases separately: Theorem \ref{t:zN_weakconv_nonsinuglar} considers situations in which $\Sigma_0$ (the covariance matrix of $F(x_0,\xi)$) is nonsingular, and Theorem \ref{t:zN_weakconv_sinuglar} handles situations in which $\Sigma_0$ is singular. The confidence region is given in \eqref{eqn:conf_reg_nonsingular} or \eqref{eqn:R_n_epsilon} for the two cases respectively.
We use $\chi^2_l$ to denote a $\chi^2$ random variable with $l$ degrees of freedom, and use $\chi^2_l(\alpha)$ to denote the number that satisfies $P( \chi^2_l > \chi^2_l(\alpha)) = \alpha$ for $\alpha\in [0,1]$. For the rest of this section, let $\alpha\in [0,1]$ be fixed.

\begin{theorem}\label{t:zN_weakconv_nonsinuglar}
Suppose that Assumptions \ref{assu1} and \ref{assu2} hold, and that $\Sigma_0$ is nonsingular. For almost every $\w\in \Omega$, there exists an integer $N_\w$ such that $\Sigma_n$ is nonsingular for $n\ge N_\w$. Moreover,
\begin{equation}\label{eqn:dfNzN-weakconv}
-\sqrt{n}\Sigma_n^{-1/2}[d(f_n)_S(z_n) (z_0 - z_n)] \Rightarrow \mN(0, I_q),
\end{equation}
and the probability for $z_0$ to belong to the set
\begin{equation}\label{eqn:conf_reg_nonsingular}
\left\{z\in \Rset^q \left|
 n \big[ d(f_n)_S(z_n) (z - z_n) \big]^T
 \Sigma_n^{-1}
 \big[d(f_n)_S(z_n) (z - z_n)\big]
  \le \chi^2_{q}(\alpha)\right.
  \right\}
\end{equation}
converges to $1-\alpha$ as $n\to \infty$.
\end{theorem}

\begin{proof}
Since $\Sigma_n$ converges to $\Sigma_0$ almost surely (Lemma \ref{l:Sigma_n_Sigma_0}), it is nonsingular for sufficiently large $n$ almost surely. Equation \eqref{eqn:dfNzN-weakconv} follows from \eqref{eqn:zN_dist_cov_dfN}. From \eqref{eqn:dfNzN-weakconv} the random variable
\[n \big[ d(f_n)_S(z_n) (z_0 - z_n) \big]^T
 \Sigma_n^{-1}
 \big[d(f_n)_S(z_n) (z_0 - z_n)\big]\]
weakly converges to a $\chi^2_q$ random variable, so  the set \eqref{eqn:conf_reg_nonsingular} contains $z_0$ with probability converging to $1-\alpha$.
\end{proof}

\begin{theorem}\label{t:zN_weakconv_sinuglar}
Suppose that Assumptions \ref{assu1} and \ref{assu2} hold, and that $\Sigma_0$ is singular. Let $\rho>0$ be the minimum of all positive eigenvalues of $\Sigma_0$, and let $l$ be the number of positive eigenvalues of $\Sigma_0$ counted with regard to their algebraic multiplicities. Let $\rho_0$ satisfy $0<\rho_0<\rho$. Decompose $\Sigma_n$ as
\begin{equation}\label{eqn:decom_Sigman}
\Sigma_n=U_n^T \Delta_n U_n
\end{equation}
where $U_n$ is an orthogonal $q\times q$ matrix, and $\Delta_n$ is a diagonal matrix with monotonically decreasing elements.
Let $D_n$ be the upper-left submatrix of $\Delta_n$ whose diagonal elements are at least $\rho_0$.
Let $l_n$ be the number of rows in $D_n$, $(U_n)_1$ be the submatrix of $U_n$ that consists of its first $l_n$ rows, and $(U_n)_2$ be the submatrix that consists of the remaining rows of $U_n$.
Then for almost every $\w$ the equality $l_n=l$ holds for sufficiently large $n$. Moreover,
\begin{equation}\label{eqn:zN_dist_cov_chi2}
 n \big[ d(f_n)_S(z_n) (z_0 - z_n) \big]^T
(U_n)_1^T D_n^{-1}(U_n)_1
 \big[d(f_n)_S(z_n) (z_0 - z_n)\big] \Rightarrow \chi^2_l
\end{equation}
and
\begin{equation}\label{eqn:zN_dist_cov_0}
n [d(f_n)_S(z_n) (z_0 - z_n)]^T (U_n)_2^T (U_n)_2 [d(f_n)_S(z_n) (z_0 - z_n)] \Rightarrow 0.
\end{equation}
For each $\epsilon>0$, the set
\begin{equation}\label{eqn:R_n_epsilon}
R_{n,\epsilon}=
\left\{z\in \Rset^q \left|
\begin{array}{l}
 n \big[ d(f_n)_S(z_n) (z - z_n) \big]^T
(U_n)_1^T D_n^{-1}(U_n)_1
 \big[d(f_n)_S(z_n) (z - z_n)\big]
  \le \chi^2_{l_n}(\alpha) \\ \\
 \|\sqrt{n}(U_n)_2 [d(f_n)_S(z_n) (z - z_n)]\|_\infty
 \le \epsilon \end{array}\right.
\right\}
\end{equation}
contains $z_0$ with probability converging to $1-\alpha$.
\end{theorem}

\begin{proof}
Let $\rho$ and $l$ be as defined in this theorem, and conduct an eigen-decomposition of $\Sigma_0$ as
\begin{equation}\label{eqn:decom_Sigma0}
\Sigma_0=U_0^T \begin{bmatrix} D_0 & 0 \\ 0 & 0 \end{bmatrix} U_0 = \begin{bmatrix} (U_0)_1^T & (U_0)_2^T \end{bmatrix} \begin{bmatrix} D_0 & 0 \\ 0 & 0 \end{bmatrix} \begin{bmatrix} (U_0)_1 \\ (U_0)_2 \end{bmatrix}
\end{equation}
where $U_0$ is orthogonal, $D_0$ is diagonal with monotonically decreasing positive diagonal elements, and $(U_0)_1$ and $(U_0)_2$ contains the first $l$ and the last $q-l$ rows of $U_0$ respectively. From \eqref{eqn:zN_dist_cov_dfN} we have
\begin{equation}\label{eqn:zN-weakconv3}
-\sqrt{n}\begin{bmatrix} D_0^{-1/2} & 0 \\ 0 & I_{q-l} \end{bmatrix}U_0
[d(f_n)_S(z_n) (z_0 - z_n)] \Rightarrow \mN(0, I_l) \times 0.
\end{equation}

From Lemma \ref{l:Sigma_n_Sigma_0}, for a.e. $\w\in\Omega$ there exists an integer $N^1_\w$, such that each $n \ge N^1_\w$ satisfies
\begin{equation}\label{eqn:deltaN_ii}
   (\Delta_n)_{ii} > \rho_0 \text{ for each } i=1,\cdots, l \text{ and }
(\Delta_n)_{ii} < \rho_0 \text{ for each } i=l+1,\cdots, q.
\end{equation}
It follows that that $l_n=l$ for $n\ge N^1_\w$ and that $D_n$ converges to $D_0$ almost surely.

For $n\ge N^1_\w$, define
\[
\hat{\Sigma}_n=U_n^T \begin{bmatrix} D_n & 0 \\ 0 & 0\end{bmatrix} U_n \text{ and }
\hat{\Sigma}_n^+=U_n^T \begin{bmatrix} D_n^{-1} & 0 \\ 0 & 0\end{bmatrix} U_n=(U_n)_1^T D_n^{-1} (U_n)_1.\]
Because of \eqref{eqn:deltaN_ii}, $\hat{\Sigma}_n$ is the unique best approximation of $\Sigma_n$ in Frobenius norm among all matrices of rank $\le l$; see, e.g., \cite{chi:ppe,gol.hof.ste:gey,ste:ehs}.
 Since the unique best approximation of a matrix depends continuously on that matrix (by an application of \cite[Theorem 1.17]{rtr.wet:va}), and the best approximation of $\Sigma_0$ is itself,  $\hat{\Sigma}_n$ almost surely converges to $\Sigma_0$. As the pseudo-inverse of $\hat{\Sigma}_n$, $\hat{\Sigma}_n^+$ almost surely converges to the pseudo-inverse of $\Sigma_0$ (by an application of \cite[Corollary 3.5]{ste:ppi}).
This and \eqref{eqn:zN-weakconv3} imply \eqref{eqn:zN_dist_cov_chi2}.

Finally, an application of \cite[Theorem 3.1]{li:rpt} implies that the angle between the row spaces of $(U_0)_2$ and $(U_n)_2$ almost surely converges to zero as $n\to \infty$. In view of \cite[Equation (2.4)]{li:rpt} and \cite[Theorem 2.6.1]{gol.van:mc}, the matrix $(U_n)_2^T (U_n)_2$ converges to $(U_0)_2^T (U_0)_2$ almost surely. This and \eqref{eqn:zN-weakconv3} imply \eqref{eqn:zN_dist_cov_0}.

Note that the choice of the matrix $U_n$ in the decomposition of $\Sigma_n$ is non-unique, when $\Sigma_n$ has repeated eigenvalues. However, when  \eqref{eqn:deltaN_ii} holds the matrices $(U_n)_1^T D_n^{-1}(U_n)_1$ and $(U_n)_2^T (U_n)_2$ that appear in \eqref{eqn:zN_dist_cov_chi2} and \eqref{eqn:zN_dist_cov_0} are uniquely determined by $\Sigma_n$, and depend continuously on $\Sigma_n$.


Since $l_n=l$ almost surely for sufficiently large $n$, from \eqref{eqn:zN_dist_cov_chi2} we have
\begin{equation}\label{eqn:UN_DN_UN}
\lim_{n\to\infty}
\Prob\left\{
 n \big[ d(f_n)_S(z_n) (z_0 - z_n) \big]^T
(U_n)_1^T D_n^{-1}(U_n)_1
 \big[d(f_n)_S(z_n) (z_0 - z_n)\big]
  \le \chi^2_{l_n}(\alpha)\right\}=1-\alpha.
\end{equation}
By \eqref{eqn:zN_dist_cov_0} we have
\[
\lim_{n\to\infty}
\Prob\left\{
 n [d(f_n)_S(z_n) (z_0 - z_n)]^T (U_n)_2^T (U_n)_2 [d(f_n)_S(z_n) (z_0 - z_n)]
  \le \epsilon \right\}=1
\]
for each $\epsilon>0$. Since all norms are equivalent in a finite-dimensional space, we can rewrite the above equality as
\begin{equation}\label{eqn:UN_norm_infty}
\lim_{n\to\infty}
\Prob\left\{
 \sqrt{n} \|(U_n)_2 [d(f_n)_S(z_n) (z_0 - z_n)]\|_\infty
  \le \epsilon \right\}=1,
\end{equation}
where $\|\cdot\|_\infty$ denotes the $\infty$-norm.
In writing \eqref{eqn:UN_norm_infty}, we assume that $U_n$ is a measurable function of $\Sigma_n$ so that the set in consideration is a measurable set in $\Omega$. Equations \eqref{eqn:UN_DN_UN} and \eqref{eqn:UN_norm_infty} imply that the set $R_{n,\epsilon}$ contains $z_0$ with probability converging to $1-\alpha$.
\end{proof}

Since for each $\epsilon>0$ the set $R_{n,\epsilon}$ is an asymptotically exact $(1-\alpha)100\%$ confidence region for $z_0$, it is natural to ask whether the same is true about the set
\begin{equation}\label{eqn:R_n_0}
R_{n,0}=
\left\{z\in \Rset^q \left|
\begin{array}{l}
 n \big[ d(f_n)_S(z_n) (z - z_n) \big]^T
(U_n)_1^T D_n^{-1}(U_n)_1
 \big[d(f_n)_S(z_n) (z - z_n)\big]
  \le \chi^2_{l_n}(\alpha) \\ \\
 \sqrt{n}(U_n)_2 [d(f_n)_S(z_n) (z - z_n)]
  =0
 \end{array}\right.
\right\}.
\end{equation}
We cannot prove that $R_{n,0}$ contains $z_0$ with probability converging to $1-\alpha$, because the quantity
\[
\lim_{n\to\infty}
\Prob\left\{
 \sqrt{n} (U_n)_2 [d(f_n)_S(z_n) (z_0 - z_n)]
  = 0 \right\}
\]
may not equal 1. However, Proposition \ref{p:R_0} below shows for a fixed $n$ that the set-valued map $R_{n,\epsilon}$ is Lipschitz continuous on an interval $[0,\bar{\epsilon}]$ for some positive real number  $\bar{\epsilon}$. Accordingly, any open set that contains $R_{n,0}$ must include $R_{n,\epsilon}$ for a sufficiently small $\epsilon$. Since $R_{n,0}$ has a simpler geometric structure comparing to $R_{n,\epsilon}$, one may choose to use $R_{n,0}$ to approximate $R_{n,\epsilon}$ for computational efficiency.

\begin{proposition}\label{p:R_0}
Let $d(f_n)_S(z_n)$ be a fixed invertible linear function, $U_n$ a fixed orthogonal matrix, and $D_n$ a fixed $l_n\times l_n$ diagonal matrix with positive diagonal elements. There exist real numbers $\bar{\epsilon}>0$ and $\sigma>0$ such that
\[R_{n,\epsilon'} \subset R_{n,\epsilon} + \sigma |\epsilon-\epsilon'|\mB\]
whenever $0\le \epsilon \le \bar{\epsilon}$ and $0\le \epsilon' \le \bar{\epsilon}$, where $\mB$ denotes the closed unit ball in $\Rset^q$.
\end{proposition}

\begin{proof}
After the coordinate transformation $U_n d(f_n)_S(z_n)$, one can view $R_{n,\epsilon}$ as a cartesian product of a fixed ellipsoid in $\Rset^{l_n}$ and a polyhedron in $\Rset^{q-l_n}$, with $\epsilon$ being the right hand side of the linear constraints defining the polyhedron. The Lipschitz continuity can be seen in the transformed space.
\end{proof}

It is shown in Proposition \ref{p:invertible} that $d(f_n)_S(z_n)$ is an invertible linear function with high probability. When this holds, the set in \eqref{eqn:conf_reg_nonsingular} and $R_{n,0}$ in \eqref{eqn:R_n_0} are ellipsoids in $\Rset^q$, and therefore can be described using their centers, principal directions, and semi-axes.
It is often desirable to provide simultaneous confidence intervals for components of $z_0$, as intervals are more convenient to describe, visualize, and interpret. We can compute these intervals by finding the minimum enclosing box of a confidence region, i.e., by computing the maximal and minimal values of $z_i$ for each $i=1,\cdots,q$ over the confidence region.

Since $\Sigma_0$ is unknown in practice, it is often difficult to directly check if $\Sigma_0$ is nonsingular or not. Instead, we conduct an eigen-decomposition of the sample covariance matrix $\Sigma_n$ as in \eqref{eqn:decom_Sigman}, and partition matrices $U_n$ and $\Delta_n$ based on positive or zero eigenvalues. If all eigenvalues are strictly positive (larger than a prescribed tolerance), then we use \eqref{eqn:conf_reg_nonsingular} as the confidence region. Otherwise, we use $R_{n,\epsilon}$ in \eqref{eqn:R_n_epsilon} for some positive $\epsilon$ or $R_{n,0}$ as the confidence region. In the latter case, the confidence region is still bounded because $d(f_n)_S(z_n)$ is invertible, and is flat in the directions represented by rows of $(U_n)_2[d(f_n)_S(z_n)]$. In particular, the set $R_{n,0}$ is a degenerate ellipsoid. One can still find the minimal enclosing box of the confidence region to obtain the simultaneous confidence intervals. 
Numerical tests with singular covariance matrices are conducted in \cite{lam.lu:aac}. Lastly, we mention that the only requirement on  $\Sigma_n$ for Theorems \ref{t:zN_weakconv_nonsinuglar} and \ref{t:zN_weakconv_sinuglar} to hold is its almost sure convergence to $\Sigma_0$. Hence, any estimator of $\Sigma_0$ that converges to it almost surely can be used as $\Sigma_n$, not necessarily the sample covariance matrix.

\section{Individual confidence intervals}\label{s:ind_ci}

This section discusses computation of individual confidence intervals for $z_0$.
An individual confidence interval for the $j$th component of $z_0$ is an interval in $\Rset$ that contains $(z_0)_j$ with a prescribed level of confidence. It is generally narrower than the simultaneous confidence interval for the same level. The difference between widths of the two intervals are substantial in problems with moderate or large dimensions. In such problems, the minimum enclosing box of a confidence region can be much larger than the region itself, resulting in simultaneous confidence intervals too wide to be useful. The sizes of individual confidence intervals are less affected by the dimension of the problem. There are also problems where only confidence intervals of selected components in the true solution are of interest.

Computation of individual confidence intervals requires knowledge on the distribution of each individual component of $z_n$. As shown in \eqref{eqn:zN_dist_cov}, $\sqrt{n}(z_n-z_0)$ weakly converges to the distribution of a random variable $\Gamma=(L_K) ^{-1} (Y_0)$. Given results in the preceding two sections, it is natural to ask whether we can substitute $L_K=d (f_0)_S (z_0)$ with $d(f_n)_S(z_n)$ in computing individual confidence intervals for $z_0$. Such a direct substitution results in a confidence interval in \eqref{eqn:ind_ci_conv_formula}, where $r_{nj}$ is defined in \eqref{eqn:def_rNj}. The equation \eqref{eqn:ind_ci_conv_general} expresses the limiting probability for such an interval to contain $(z_0)_j$ in terms of  $\Gamma$. That limiting probability equals the desired confidence level $1-\alpha$ when the condition in \eqref{eqn:prob_intersection} holds.

Let $\mP(z_0)=\{P_1,\cdots,P_k\}$ be the set of $q$-cells in the normal manifold of $S$ that contains $z_0$. For each $i=1,\cdots,k$, let $K_i=\cone(P_i-z_0)$, $A_i$ be the matrix representing $d\Pi_S(z)$ for points $z$ in the interior of $P_i$, and $T_i=d (f_0)_S (z_0) (K_i)$ be the image of $K_i$ under the function $d(f_0)_S(z_0)$. From \eqref{eqn:d_f0_S_h} we can see that $d (f_0)_S(z_0)$ coincides with
$d L_S(z_0)$, where $L_S$ is the normal map induced by the linear operator $L$ and the set $S$. On the cone $K_i$, the map $d(f_0)_S(z_0)$  is represented by the matrix
\[
 M_i = L A_i + I - A_i.
 \]
Under Assumption \ref{assu2} the map $d(f_0)_S(z_0)$ is a global homeomorphism, so all the matrices $M_i, i=1,\cdots, k$ are nonsingular.

As above let $\Gamma=(L_K) ^{-1} (Y_0)$, and for each $i=1,\cdots,k$ define a random variable $\Gamma^i=M_i ^{-1} (Y_0)$. The fact that $Y_0$ is a multivariant normal random variable implies that each $\Gamma^i$ is a multivariant normal random variable with covariance matrix $M_i^{-1} \Sigma_0 M_i^{-T}$. Accordingly, each component of $\Gamma^i$ is a normal random variable. If we define a number
 \[
 r^i_j = \sqrt{(M_i^{-1} \Sigma_0 M_i^{-T})_{jj}}
 \]
for each $i=1,\cdots,k$ and $j=1,\cdots,q$, then for each real number $\alpha\in (0,1)$ we have
 \begin{equation}\label{eqn:r_ij_prob}
 \Prob\left(|\Gamma^i_j|\le \sqrt{\chi^2_1(\alpha)} r^i_j\right) = 1-\alpha.
 \end{equation}

Let us discuss a relation between $\Gamma$ and $\Gamma^i$. Since $d(f_0)_S(z_0)$ is represented by the matrix $M_i$ on the cone $K_i$, we have
\[
d (f_0)_S (z_0) ^{-1} (y)= M_i^{-1}(y) \text{ if } y\in T_i.
\]
For each measurable set $W \subset K_i$ we have
\begin{equation}\label{eqn:prob_gamma_W}
\Prob(\Gamma \in W)= \Prob(Y_0\in M_i (W)) = \Prob(\Gamma^i \in W).
\end{equation}
%
%

\begin{theorem}\label{t:ind_ci}
Suppose that Assumptions \ref{assu1} and \ref{assu2} hold, and suppose that $\Sigma_0$ has a positive determinant. Let $P_i, K_i, T_i, A_i, M_i, \Gamma^i, \Gamma, r^i_j$ be defined as above.
For each integer $n$ and each $j=1,\cdots,q$, define a number
\begin{equation}\label{eqn:def_rNj}
r_{nj} =
\left\{
\begin{array}{ll}
\sqrt{(d(f_n)_S(z_n)^{-1} \Sigma_n d(f_n)_S(z_n)^{-T})_{jj}} & \text{ if } d(f_n)_S(z_n) \text{ is an invertible linear map,}\\
0 & \text{ otherwise.}
\end{array}
\right.
\end{equation}
 Then for each real number $\alpha \in (0,1)$ and for each $j=1,\cdots,q$,
 \begin{equation}\label{eqn:ind_ci_conv_general}
\begin{split}
&\lim_{n\to\infty} \Prob\left(\frac{\sqrt{n}|(z_n-z_0)_j|}{r_{nj}}\le \sqrt{\chi^2_1(\alpha)}\right)\\
=&\sum_{i=1}^k \Prob\left(\big|\frac{\Gamma^i_j}{r^i_j}\big|\le \sqrt{\chi^2_1(\alpha)}\text{ and } \Gamma^i\in K_i\right)
=\sum_{i=1}^k \Prob\left(\big|\frac{\Gamma_j}{r^i_j}\big|\le \sqrt{\chi^2_1(\alpha)}\text{ and } \Gamma\in K_i\right).
\end{split}
\end{equation}
Moreover, suppose for a given $j=1,\cdots,q$ that the following equality
\begin{equation}\label{eqn:prob_intersection}
\Prob \left(\big|\frac{\Gamma^i_j}{r^i_j}\big|\le \sqrt{\chi^2_1(\alpha)}\text{ and } \Gamma^i\in K_i\right) = \Prob \left(\big|\frac{\Gamma^i_j}{r^i_j}\big|\le \sqrt{\chi^2_1(\alpha)}\right) \Prob (\Gamma^i\in K_i)
\end{equation}
holds for each $i=1,\cdots,k$. Then for each real number $\alpha \in (0,1)$,
\begin{equation}\label{eqn:ind_ci_conv}
\lim_{n\to \infty} \Prob\left\{|(z_n-z_0)_j|\le \frac{\sqrt{\chi^2_1(\alpha)} r_{nj}}{\sqrt{n}}\right\} = 1-\alpha.
\end{equation}
\end{theorem}

\begin{proof}
Let $\alpha \in (0,1)$ be fixed.
Recall from Lemma \ref{l:Sigma_n_Sigma_0} that $\Sigma_n$ converges to $\Sigma_0$ almost surely, and from the proof of Theorem \ref{t:d_f0_S_d_fN_S} that $d f_n(x_n)$ converges to $L=d f_0(x_0)$ almost surely. Let $Z'$ be a neighborhood of $z_0$ in the interior of $|\mP(z_0)|$ with $Z'\cap P_i = Z'\cap (z_0+K_i)$ for each $i=1,\cdots,k$. We have
\begin{equation}\label{eqn:prob_zN_Z'}
\lim_{n\to \infty} \Prob(z_n\in Z')=1.
\end{equation}
 Writing $Z'$ as the following union,
\[Z'=\cup_{i=1}^k (Z'\cap (z_0+\Int K_i)) \cup (Z'\cap (z_0+\Rset^q/\cup_{i=1}^k \Int K_i)),\]
and noting from \eqref{eqn:prob_R} in the proof of Proposition \ref{p:invertible} that
\[\lim_{n\to\infty}\Prob(z_n\in Z'\cap (z_0+\Rset^q/\cup_{i=1}^k \Int K_i))=0,\]
we find
\begin{equation}\label{eqn:prob_zN_intKi}
\lim_{n\to \infty} \sum_{i=1}^k \Prob(z_n\in Z'\cap (z_0+\Int K_i)) =1.
\end{equation}
It is not hard to see $Z'\cap (z_0+\Int K_i) = Z' \cap \Int P_i$ for each $i=1,\cdots,k$. When $z_n \in Z'\cap \Int P_i$, $d (f_n)_S(z_n)$ is a linear map given by
\[
d (f_n)_S (z_n) = d f_n (x_n) A_i + I - A_i.
\]

For each integer $n$, $j=1,\cdots,q$ and $i=1,\cdots,k$, define a quantity
\begin{equation}\label{eqn:def_hatr}
\hat{r}^i_{nj}=r_{nj} 1_{z_n \in Z'\cap \Int P_i} + r^i_j 1_{z_n \not\in Z'\cap \Int P_i}.
\end{equation}
By Proposition \ref{p:invertible}, and from the fact that $d f_n(x_n)$ and $\Sigma_n$ converge to $d f_0(x_0)$ and $\Sigma_0$ almost surely respectively, $\hat{r}^i_{nj}$ converges to $r^i_j$ in probability as $n\to \infty$.

Now, let us first consider the situations in which $k\ge 2$. These are the situations in which $z_0$ lies on the boundary of some $q$-cell.
For each $i=1,\cdots,k$ and $j=1,\cdots,q$, choose an $q$-dimensional vector $\bar{h}^{ij}$ such that $\bar{h}^{ij}$ does not belong to $K_i$ and its $j$th component $\bar{h}^{ij}_j$ satisfies $|\bar{h}^{ij}_j|>r^i_j \sqrt{\chi^2_1(\alpha)}$. Define a random variable $h^{ij}_{n}\in \Rset^q$ by
\begin{equation}\label{eqn:def_h_ijN}
h^{ij}_{n}=\sqrt{n}(z_n-z_0)1_{z_n \in Z'\cap \Int P_i} + \bar{h}^{ij} 1_{z_n \not \in Z'\cap \Int P_i},
\end{equation}
and define another $q$-dimensional random variable $\hat{\Gamma}^{ij}$ by
\begin{equation}\label{eqn:def_Gamma_ij}
\hat{\Gamma}^{ij} = \Gamma^i 1_{\Gamma^i\in \Int K_i} + \bar{h}^{ij} 1_{\Gamma^i \not\in \Int K_i}.
\end{equation}
Let $W$ be a measurable subset of $\Int K_i$ with $\Prob(\Gamma\in \partial W)=0$, where $\partial W$ stands for the boundary of $W$. The above definition of $h^{ij}_n$ and the fact that $\bar{h}^{ij}$ does not belong to $K_i$ imply
\[
\Prob(h^{ij}_n \in W)=\Prob(\sqrt{n}(z_n-z_0) \in W \text{ and } z_n \in Z'\cap \Int P_i).
\]
Recalling that $Z'\cap \Int P_i = Z' \cap (z_0+\Int K_i)$ for each $i=1,\cdots,k$, we find
\[
\Prob(\sqrt{n}(z_n-z_0) \in W \text{ and } z_n \in Z'\cap \Int P_i)=\Prob(\sqrt{n}(z_n-z_0) \in W \text{ and } z_n \in Z').
\]
Combining the above two equalities with \eqref{eqn:prob_zN_Z'}, we have
\[
\lim_{n\to\infty} \Prob(h^{ij}_n \in W)=\lim_{n\to\infty} \Prob(\sqrt{n}(z_n-z_0) \in W).
\]
By the asymptotic distribution \eqref{eqn:zN_dist_cov}, the definition of $\Gamma$ and equation \eqref{eqn:prob_gamma_W}, we have
\[\lim_{n\to\infty} \Prob(\sqrt{n}(z_n-z_0) \in W)
=\Prob(\Gamma\in W)=\Prob(\Gamma^i\in W).\]
From the definition of $\hat{\Gamma}^{ij}$ in \eqref{eqn:def_Gamma_ij} and the facts $W\subset \Int K_i$ and $\bar{h}^{ij}\not\in K_i$ we have
\[
\Prob(\Gamma^i\in W)=\Prob(\hat{\Gamma}^{ij}\in W).
\]
Combining the above three equalities together, we find
\[
\lim_{n\to\infty} \Prob(h^{ij}_n \in W)=\Prob(\hat{\Gamma}^{ij}\in W)\]
for each measurable set $W\subset \Int K_i$ with $\Prob(\Gamma\in \partial W)=0$.  Because the set $\Int K_i$ itself satisfies $\Prob(\Gamma\in \partial (\Int K_i))=0$, the above equality holds with $\Int K_i$ in place of $W$. Since $h^{ij}_n$ and $\hat{\Gamma}^{ij}$ only take values in $\Int K_i \cup \{\bar{h}^{ij}\}$, we have $\lim_{n\to\infty} \Prob(h^{ij}_n = \bar{h}^{ij})=\Prob(\hat{\Gamma}^{ij}=\bar{h}^{ij})$. Also, for each $W\subset \Int K_i$ we have $\Prob(\Gamma\in \partial W)=\Prob(\hat{\Gamma}^{ij}\in \partial W)$. It is not hard to see
\[h^{ij}_n \Rightarrow \hat{\Gamma}^{ij}.\] Since $\hat{r}^i_{nj}$ converges in probability to the fixed number $r^i_j$, which is strictly positive under the assumption in this theorem about $\Sigma_0$, we have
\[
\frac{(h^{ij}_n)_j}{\hat{r}^i_{nj}}\Rightarrow \frac{\hat{\Gamma}^{ij}_j}{r^i_j},
\]
where $(h^{ij}_n)_j$ and $\hat{\Gamma}^{ij}_j$ are the $j$th components of $(h^{ij}_n)$ and $\hat{\Gamma}^{ij}$ respectively.
It follows that
\begin{equation}\label{eqn:h_Nij_hatGamma}
\lim_{n\to\infty} \Prob\left(\big|\frac{(h^{ij}_n)_j}{\hat{r}^i_{nj}}\big|\le \sqrt{\chi^2_1(\alpha)}\right) = \Prob\left(\big|\frac{\hat{\Gamma}^{ij}_j}{r^i_j}\big|\le \sqrt{\chi^2_1(\alpha)}\right),
\end{equation}
because the probability for $\frac{\hat{\Gamma}^{ij}_j}{r^i_j}$ to lie on the boundary of $[-\sqrt{\chi^2_1(\alpha)},\sqrt{\chi^2_1(\alpha)}]$ is zero.
The way $\hat{\Gamma}^{ij}$ is defined in \eqref{eqn:def_Gamma_ij} and the fact that $|\bar{h}^{ij}_j|>r^i_j \sqrt{\chi^2_1(\alpha)}$ imply
\begin{equation}\label{eqn:hatGamma_Gamma}
\Prob\left(\big|\frac{\hat{\Gamma}^{ij}_j}{r^i_j}\big|\le \sqrt{\chi^2_1(\alpha)}\right) = \Prob\left(\big|\frac{\Gamma^i_j}{r^i_j}\big|\le \sqrt{\chi^2_1(\alpha)}\text{ and } \Gamma^i\in \Int K_i\right).
\end{equation}
The facts that $|\bar{h}^{ij}_j|>r^i_j \sqrt{\chi^2_1(\alpha)}$ and
that $\hat{r}^i_{nj}$ almost surely converges to $r^i_j$
imply
\[\lim_{n\to\infty} \Prob\left(\big|\frac{\bar{h}^{ij}_j}{\hat{r}^i_{nj}}\big|\le \sqrt{\chi^2_1(\alpha)}\right)=0.\]

We are now ready to put all pieces together to prove \eqref{eqn:ind_ci_conv_general} for the case $k\ge 2$. By the definition of $h^{ij}_n$ in \eqref{eqn:def_h_ijN} and the above equality, we have
\begin{equation}\label{eqn:h_Nij_zN}
\begin{split}
&\lim_{n\to\infty} \Prob\left(\big|\frac{(h^{ij}_n)_j}{\hat{r}^i_{nj}}\big|\le \sqrt{\chi^2_1(\alpha)}\right)
\\
=&\lim_{n\to\infty} \Prob\left(\frac{\sqrt{n}|(z_n-z_0)_j|}{\hat{r}^i_{nj}}\le \sqrt{\chi^2_1(\alpha)} \text{ and } z_n \in Z'\cap \Int P_i\right).
\end{split}
\end{equation}
By the definition of $\hat{r}^i_{nj}$ in \eqref{eqn:def_hatr}, we can replace it by $r_{nj}$ in the right hand side of \eqref{eqn:h_Nij_zN}.
Combining \eqref{eqn:h_Nij_hatGamma}, \eqref{eqn:hatGamma_Gamma} and \eqref{eqn:h_Nij_zN} together, we have
\begin{equation}
\begin{split}
&\lim_{n\to\infty} \Prob\left(\frac{\sqrt{n}|(z_n-z_0)_j|}{r_{nj}}\le \sqrt{\chi^2_1(\alpha)} \text{ and } z_n \in Z'\cap \Int P_i\right)\\
=&\Prob\left(\big|\frac{\Gamma^i_j}{r^i_j}\big|\le \sqrt{\chi^2_1(\alpha)}\text{ and } \Gamma^i\in \Int K_i\right)\\
=&\Prob\left(\big|\frac{\Gamma^i_j}{r^i_j}\big|\le \sqrt{\chi^2_1(\alpha)}\text{ and } \Gamma^i\in K_i\right)
\end{split}
\end{equation}
where the second equality follows from the fact that the probability for $\Gamma^i$ to belong to the boundary of $K_i$ is 0.
Recalling that $Z'\cap (z_0+\Int K_i)=Z'\cap \Int P_i$, we can combine the above equality with \eqref{eqn:prob_zN_intKi} to prove \eqref{eqn:ind_ci_conv_general}.

 Under the assumption \eqref{eqn:prob_intersection}, we have
\begin{equation}\label{eqn:sum_Prob_Gamma_Ki}
\begin{split}
&\sum_{i=1}^k \Prob\left(\big|\frac{\Gamma^i_j}{r^i_j}\big|\le \sqrt{\chi^2_1(\alpha)} \text{ and } \Gamma^i\in K_i\right)\\
=&\sum_{i=1}^k \Prob\left(\big|\frac{\Gamma^i_j}{r^i_j}\big|\le \sqrt{\chi^2_1(\alpha)}\right)\Prob\left( \Gamma^i\in K_i\right)\\
=&\sum_{i=1}^k \Prob\left(\big|\frac{\Gamma^i_j}{r^i_j}\big|\le \sqrt{\chi^2_1(\alpha)}\right)\Prob\left( \Gamma\in K_i\right) = 1-\alpha,
\end{split}
\end{equation}
where the third equality uses \eqref{eqn:prob_gamma_W} and the fourth equality uses \eqref{eqn:r_ij_prob} and the fact that $\sum_{i=1}^k \Prob\left( \Gamma\in K_i\right) =1$. This proves \eqref{eqn:ind_ci_conv} for the case $k\ge 2$.

The case $k=1$ is much simpler. In this case, $z_0$ lies in the interior of a single $q$-cell $P_1$, $K_1=\Rset^q$, $d(f_0)_S(z_0)$ is an invertible linear map, and $\Gamma$ is equal to $\Gamma^1$. Since $z_n$ belongs to the interior of $P_1$ for all sufficiently large $n$, the number $r_{nj}$ converges almost surely to $r^1_j$ for each $j=1,\cdots,q$. Equation \eqref{eqn:ind_ci_conv_general} follows from the fact that
\[
\frac{\sqrt{n}(z_n-z_0)_j}{r_{nj}} \Rightarrow \frac{\Gamma^1_j}{r^1_j},
\]
and equation \eqref{eqn:ind_ci_conv} follows from the equality \eqref{eqn:r_ij_prob} with $i=1$.
\end{proof}

Below, we discuss two situations in which the equality \eqref{eqn:prob_intersection} holds.

The first situation is when $k\le 2$. Obviously, when $k=1$ (that is, when $z_0$ lies in the interior of an $q$-cell in the normal manifold of $S$), the cone $K_1$ is the entire space $\Rset^q$, and the assumption \eqref{eqn:prob_intersection} automatically holds. It was noted by Michael Lamm that \eqref{eqn:prob_intersection} also holds when $k=2$. In the latter case, the cone $K_i$ is a half-space for $i=1,2$. Since $\Gamma^i$ is a multivariant normal random variable with mean zero, $-\Gamma^i$ and $\Gamma^i$ have the same distribution. It follows that $\Prob \left(\Gamma^i\in K_i\right)=1/2$, and that $\Prob \left(\big|\frac{\Gamma^i_j}{r^i_j}\big|\le \sqrt{\chi^2_1(\alpha)}\text{ and } \Gamma^i\in K_i\right)$ and $\Prob \left(\big|\frac{\Gamma^i_j}{r^i_j}\big|\le \sqrt{\chi^2_1(\alpha)}\text{ and } -\Gamma^i\in K_i\right)$ are both equal to $\frac{1}{2}\Prob \left(\big|\frac{\Gamma^i_j}{r^i_j}\big|\le \sqrt{\chi^2_1(\alpha)}\right)$. This proves \eqref{eqn:prob_intersection} when $k=2$.

The second situation is when $S$ is a box and each $\Gamma^i$ has a diagonal covariance matrix. In this case, each $K_i$ is of the form $\{h\in \Rset^q \mid h_i\ge 0 \text{ for each }i\in I_+, h_i\le 0 \text{ for each }i\in I_-\}$ for some disjoint subsets $I_+$ and $I_-$ of $\{1,\cdots,q\}$. Since each $\Gamma^i$ has a diagonal covariance matrix, its components $\Gamma^i_j$ are independent of each other. From such independence, and by the symmetry of a mean-0 normal random variable in $\Rset$ with respect to the origin, it is not hard to see that \eqref{eqn:prob_intersection} holds for every $j$ and $i$.

Whenever \eqref{eqn:prob_intersection} holds, we have \eqref{eqn:ind_ci_conv}, which means that the interval
\begin{equation}\label{eqn:ind_ci_conv_formula}
\left[(z_n)_j - \frac{\sqrt{\chi^2_1(\alpha)} r_{nj}}{\sqrt{n}},
(z_n)_j + \frac{\sqrt{\chi^2_1(\alpha)} r_{nj}}{\sqrt{n}}\right]
\end{equation}
is an asymptotically exact $(1-\alpha)$ confidence interval for $(z_0)_j$.

The numerical examples in Section \ref{s:num} with $z_0=0$ do not belong to either of the above two cases, yet the coverage rates of individual confidence intervals obtained from this method are still reasonable. From those examples we observe that even if the difference between the two sides of \eqref{eqn:prob_intersection} is large for some $i$, the difference tends to be reduced after the summation over all $i$'s. As a result, the quantity
\[
\sum_{i=1}^k \Prob\left(\big|\frac{\Gamma^i_j}{r^i_j}\big|\le \sqrt{\chi^2_1(\alpha)}\text{ and } \Gamma^i\in K_i\right)
\]
may not be very far from $1-\alpha$. 
Finally, Theorem \ref{t:ind_ci} assumes that $\Sigma_0$ is nonsingular, which implies the nonsingularity of $\Sigma_n$ for large $n$. Even if $\Sigma_n$ is singular, one can still use \eqref{eqn:ind_ci_conv_formula} to compute a confidence interval, with the caution that the coverage probability for the true solution may not be close to the prescribed level of confidence if $\Sigma_n$ is very different from $\Sigma_0$.

\section{Numerical results}\label{s:num}

Before applying the proposed method to numerical examples, we summarize how to use this method in practice.

Among the assumptions, Assumption \ref{assu1} is standard. For Assumption \ref{assu2}, instead of directly checking if the normal map $L_K$ is a homeomorphism, we can check if one of its sufficient conditions holds (see the discussion below Assumption \ref{assu2}). Likewise, we can check if \eqref{eqn:prob_intersection} holds by determining if the problem belongs to one of the two case discussed below Theorem \ref{t:ind_ci}. In checking those assumptions, the asymptotical results in \cite{lu:nmb,lu.bud:crsvi} can be used to estimate $d(f_0)_S (z_0)$ and $K$.

Confidence regions of $z_0$ are given by the set \eqref{eqn:conf_reg_nonsingular} (if $\Sigma_0$ is nonsingular) or the set $R_{n,\epsilon}$ in \eqref{eqn:R_n_epsilon} (if $\Sigma_0$ is singular), and the latter set can be approximated by $R_{n,0}$. According to Proposition \ref{p:invertible}, sets \eqref{eqn:conf_reg_nonsingular} and $R_{n,0}$ are ellipsoids with high probability. Computation of simultaneous confidence intervals for $z_0$ is done by finding the minimal bounding box of its confidence region. 
Individual confidence intervals of each component of $z_0$ can be computed using the formula (5.17), provided that \eqref{eqn:prob_intersection} holds. The quantity $r_{nj}$ defined in \eqref{eqn:def_rNj} depends on $\Sigma_n$ and $d(f_n)_S(z_n)$, and the latter is a nonsingular matrix with high probability.

To convert confidence regions and intervals of $z_0$ into those of $x_0$, the key is to use the equality $x_0=\Pi_S(z_0)$. Suppose the set $A\subset \Rset^q$ is a $(1-\alpha)100\%$ confidence region for $z_0$, then the image of $A$ under the operator $\Pi_S$, denoted by $\Pi_S(A)$, contains $x_0$ with probability at least $1-\alpha$. If $S$ is a box, then one can easily project the simultaneous confidence intervals of $z_0$ to obtain simultaneous confidence intervals of $x_0$. The latter intervals are conservative, as the probability for the product of all those intervals to contain $x_0$ is at least $1-\alpha$. When $S$ is a box, each component of $x_0$ is the projection of a component of $z_0$ onto an interval (the product of all such intervals is $S$), and one can project the individual confidence intervals of $z_0$ to obtain individual confidence intervals of $x_0$. For problems in which $S$ is not a box, the above projection method would not be easy to implement in general, and one would need to exploit special structure in those problems to obtain confidence regions and intervals of $x_0$.

\vspace{4mm}

\noindent \textbf{An example with $q=2$.} Here, we apply the method to the same example used in \cite{lu:nmb,lu.bud:crsvi}. In this example, $q=2$, $d=6$, $S=\Rset^2_+$, $F:\Rset^2\times \Rset^6 \to \Rset^2$ is defined by
\begin{equation}
F (x, \xi)= \begin{bmatrix} \xi_1 & \xi_2 \\ \xi_3 & \xi_4\end{bmatrix} \begin{bmatrix}x_1 \\ x_2 \end{bmatrix} + \begin{bmatrix}\xi_5 \\ \xi_6\end{bmatrix},
\end{equation}
and the random vector $\xi$ follows the uniform distribution over the box $[0,2]\times[0,1]\times[0,2]\times[0,4]\times[-1,1]\times[-1,1]$. The true problem is
\begin{equation}\label{eqn:lcp}
0 \in \begin{bmatrix} 1 & 1/2 \\ 1 & 2\end{bmatrix} \
x + N_{\Rset^2_+} (x).
\end{equation}
The solution to \eqref{eqn:lcp} is $x_0=0$, and the solution of the corresponding normal map formulation is $z_0 = x_0 - E[F(x_0,\xi)] = 0$. The covariance matrix of $F(x_0, \xi)= (\xi_5,\xi_6)$ is given by
\[
\Sigma_0=\begin{bmatrix}
1/3 & 0 \\
0 &1/3
\end{bmatrix},
\]
 and the B-derivative $d(f_0)_{\Rset^2_+} (z_0)$ is a piecewise linear function represented by matrices
\[
\begin{bmatrix}
  1 & 1/2 \\ 1 & 2
\end{bmatrix},
\begin{bmatrix}
       1 &    0 \\
     1   &  1
\end{bmatrix},
\begin{bmatrix}
   1 &    1/2 \\
         0   & 2
\end{bmatrix} \text{ and }
\begin{bmatrix}
     1   &  0 \\
     0   &  1
\end{bmatrix}
\]
in orthants $\Rset^2_+$, $\Rset_+\times \Rset_-$, $\Rset_-\times \Rset_+$ and $\Rset^2_-$ respectively.
Accordingly, if we define random variables $\Gamma^i$ as in Section \ref{s:ind_ci}, then the covariance matrices of them are
\[
\begin{bmatrix}
  0.6296  & -0.3704 \\
   -0.3704  &  0.2963
\end{bmatrix},
\begin{bmatrix}
    0.3333 &  -0.3333\\
   -0.3333  &  0.6667
\end{bmatrix},
\begin{bmatrix}
  0.3542  & -0.0417\\
   -0.0417  &  0.0833
\end{bmatrix} \text{ and }
\begin{bmatrix}
         0.3333    &     0\\
         0  &  0.3333
\end{bmatrix}
\]
respectively.
An SAA problem with $n=10$ is given by
\[
0 \in \begin{bmatrix}
    0.9292 &   0.5400\\
    0.7536 &   2.1111\\
\end{bmatrix}
x+
\begin{bmatrix}
   -0.1319 \\   -0.2906
\end{bmatrix} + N_{\Rset^2_+}(x).
\]
The SAA solution is $x_{10}=(0.0782, 0.1097)$, $z_{10}= ( 0.0782, 0.1097)$, and the sample covariance matrix of $F(x_{10},\xi)$ is
\[
\Sigma_{10}=\begin{bmatrix}
    0.4169  &  0.0137\\
    0.0137  &  0.1865
    \end{bmatrix}.
\]
The B-derivative $d\Pi_{\Rset^2_+}(z_{10})$ is exactly the identity map on $\Rset^2$, and the B-derivative $d (f_{10})_{\Rset^2_+}(z_{10})$ is the linear map represented by the matrix
\[
\begin{bmatrix}
    0.9292 &   0.5400\\
    0.7536 &   2.1111\\
\end{bmatrix}\begin{bmatrix} 1  & 0 \\ 0 & 1 \end{bmatrix} + \begin{bmatrix} 1  & 0 \\ 0 & 1 \end{bmatrix} - \begin{bmatrix} 1  & 0 \\ 0 & 1 \end{bmatrix} = \begin{bmatrix}
    0.9292 &   0.5400\\
    0.7536 &   2.1111
\end{bmatrix}.
\]
The confidence regions in \eqref{eqn:conf_reg_nonsingular} are given by
\[
\{z\in\Rset^2\mid 10(z-z_{10})^T
\begin{bmatrix}
 4.8810  &  9.3398 \\
     9.3398 &   24.2564 \end{bmatrix} (z-z_{10})\le \chi^2_2(\alpha)
\}.
\]

Figure \ref{f:confregz0}(a) shows boundaries of the above confidence regions. The center of these regions is $z_{10}$, marked by `$\times$' in the graph. From the innermost to the outermost, the curves correspond to boundaries of confidence regions for $z_0$ at levels 10\%, $\cdots$, 90\% respectively. The point $z_0$ is marked by `+' and lies just beyond the 90\% confidence region. The dashed rectangle shown in the figure is the minimum enclosing box of the 90\% region. Figure \ref{f:confregz0}(b) shows confidence regions for $z_0$ obtained from a different SAA problem with sample size $n=30$,
$x_{30}=0$ and $z_{30}=(   -0.0483,  -0.0114)$. Table \ref{tab:ind_simu_conf} shows the 90\% simultaneous and individual confidence intervals for $z_0$ obtained from the above two SAA problems.

\begin{figure}[ht]
\caption{Confidence regions for $z_0$ in the example $q=2$ at levels 10\%, $\cdots$, 90\%}
\centering
\mbox{
\subfigure[$n=10, z_{10}\approx( 0.08, 0.11)$]{\epsfig{file=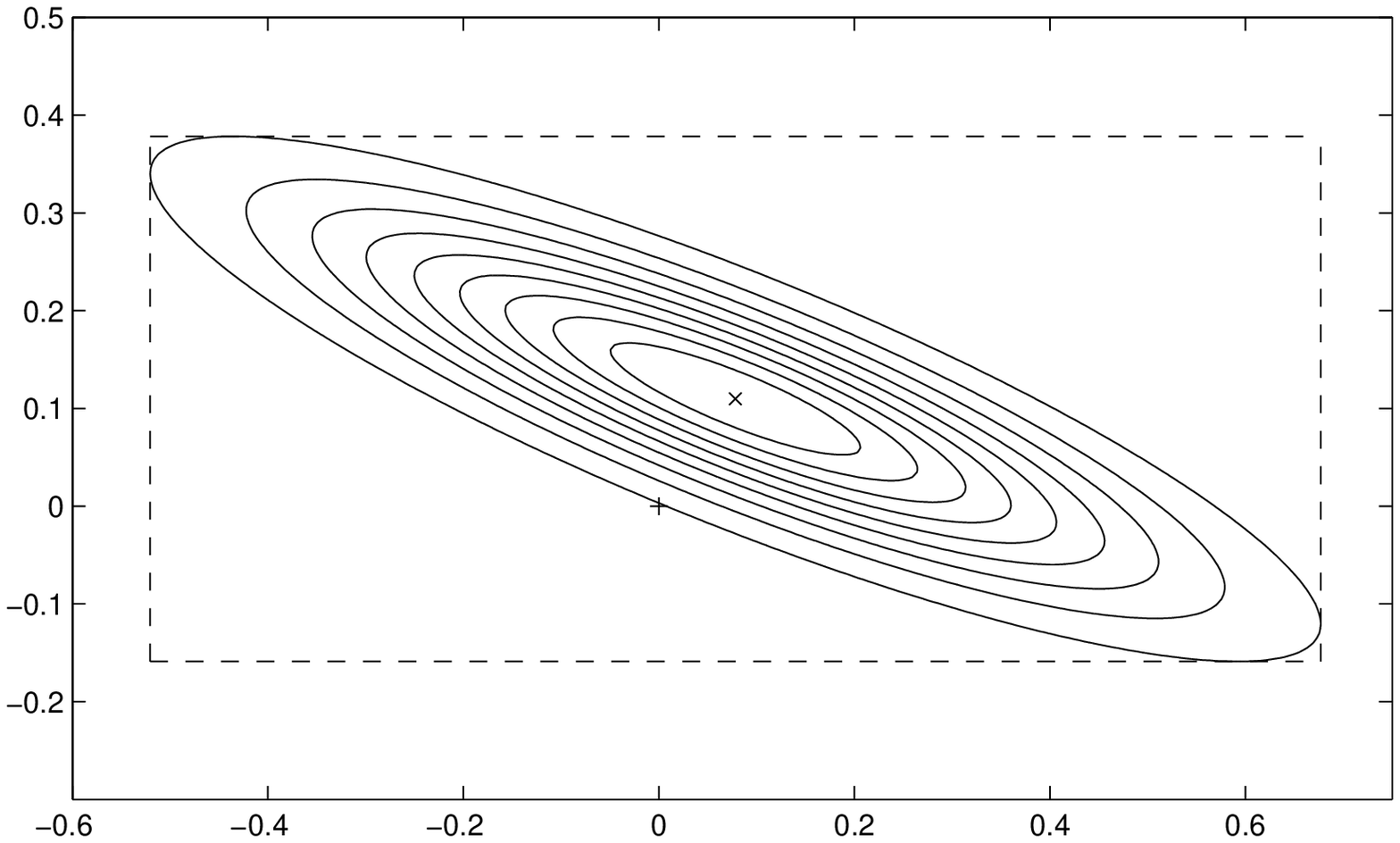,width=0.64\linewidth}}\quad
\subfigure[$n=30, z_{30}\approx( -0.05, -0.01)$]{\epsfig{file=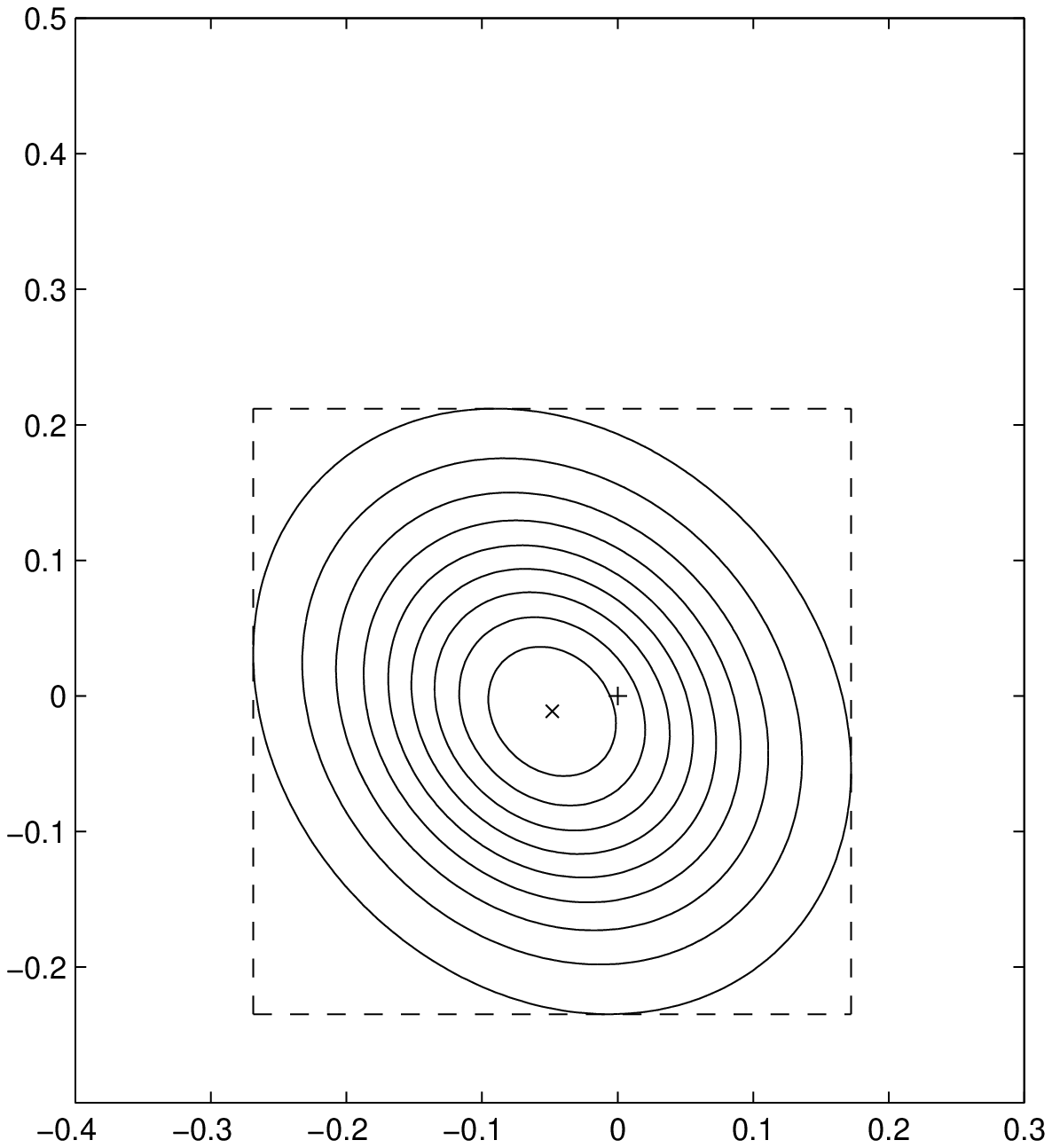,width=0.35\linewidth}}}
\label{f:confregz0}
\end{figure}

\begin{table}[ht]
\caption{Confidence intervals of level $90\%$ in the example $q=2$}
\begin{center}
\begin{tabular}{c|cll|cll}
\hline
 & \multicolumn{3}{|c|}{$n=10$} & \multicolumn{3}{|c}{$n=30$}\\
  & Est & Sim CI    &  Ind CI  & Est & Sim CI    &  Ind CI\\
\hline
$(z_0)_1$ & 0.08 & [-0.52, 0.68] & [-0.38, 0.54] & -0.05 & [-0.27, 0.17] & [-0.21, 0.12] \\
 $(z_0)_2$ & 0.11 & [-0.16, 0.38] & [-0.10, 0.32] & -0.01 & [-0.23, 0.21] & [-0.18, 0.16] \\
\hline
\end{tabular}
\end{center}
\label{tab:ind_simu_conf}
\end{table}

To test the coverage of confidence intervals obtained from the proposed method, we generate 200 SAA problems with $n=10$ and 200 SAA problems with $n=30$ from different random seeds, solve them using the PATH solver of GAMS, and compute simultaneous and individual confidence intervals for $z_0$ of levels 90\%, 95\% and 99\% from the solution for each SAA problem. We count how many times the 2-dimensional box formed by simultaneous confidence intervals cover $z_0$, and record the numbers in the row labeled by $z_0$ in Table \ref{tab:ind_conf_cov}. For example, the 90\% simultaneous confidence intervals obtained from 171 SAA problems with $n=10$ cover $z_0$ jointly. We also count how many times each component of $z_0$ is contained in the corresponding individual confidence intervals, and record the numbers in the remaining rows of Table \ref{tab:ind_conf_cov}. For example, the 90\% individual confidence intervals for $(z_0)_1$ obtained from 164 SAA problems with $n=10$ cover the true value $(z_0)_1=0$.
\begin{table}[h]
\caption{True solution coverage by confidence intervals in the example $q=2$ from 200 SAA problems}
\begin{center}
\begin{tabular}{c|ccc|ccc}
\hline
 & \multicolumn{3}{|c|}{$n=10$} & \multicolumn{3}{|c}{$n=30$}\\
     & $\alpha=$0.1 & 0.05 & 0.01  & $\alpha=$0.1 & 0.05 & 0.01  \\
\hline
 $z_0$ &  171 &  180 &  187 &  184 &  192 &  197\\ \hline
 $(z_0)_1$ &   164 &  185 &  194 &  172 &  188 &  198
 \\
 $(z_0)_2$ &    159  & 175 &  191 &    176 &  186 &   196
\\ \hline
\end{tabular}
\end{center}
\label{tab:ind_conf_cov}
\end{table}

 The proposed method generates confidence regions and simultaneous confidence intervals based on the asymptotic distribution in \eqref{eqn:dfNzN-weakconv} (or \eqref{eqn:zN_dist_cov_chi2} for singular cases). We evaluate how closely $z_n$ follow the asymptotic distribution using $\chi^2$ plots.
   We use the same SAA problems generated above for coverage tests. For each SAA problem with $n=10$ we compute the squared distance \[n \big[d(f_n)_S(z_n) (z_0 - z_n)\big]^T \Sigma_{n}^{-1} \big[d(f_n)_S(z_n) (z_0 - z_n)\big],\]
  and order these distances from smallest to largest as $d^2_{(1)}\le d^2_{(2)}\le \cdots \le d^2_{(200)}$. For each $j=1,\cdots,200$, let $q_{c,2}((j-1/2)/200)$ be the $100(j-1/2)/200$ quantile of the $\chi^2$ distribution with 2 degrees of freedom.
  We then graph the pairs $(q_{c,2}((j-1/2)/200), d^2_{(j)})$ for $j=1,\cdots, 200$ in Figure \ref{f:chi2plot}(a), in which the horizontal axis is for quantiles and the vertical axis is for squared distances. Figure \ref{f:chi2plot}(b) is obtained similarly, from the 200 SAA problems with sample size $n=30$. In each figure the points nearly follow a straight line through the origin with slope around 1, and the slope of the line in Figure \ref{f:chi2plot}(b) is closer to 1. This suggests that the expression on the left hand side of  \eqref{eqn:dfNzN-weakconv} approximately follows the standard normal distribution.

\begin{figure}
\caption{$\chi^2$ plots in the example $q=2$}
\centering
\mbox{
\subfigure[$n=10$]{\epsfig{file=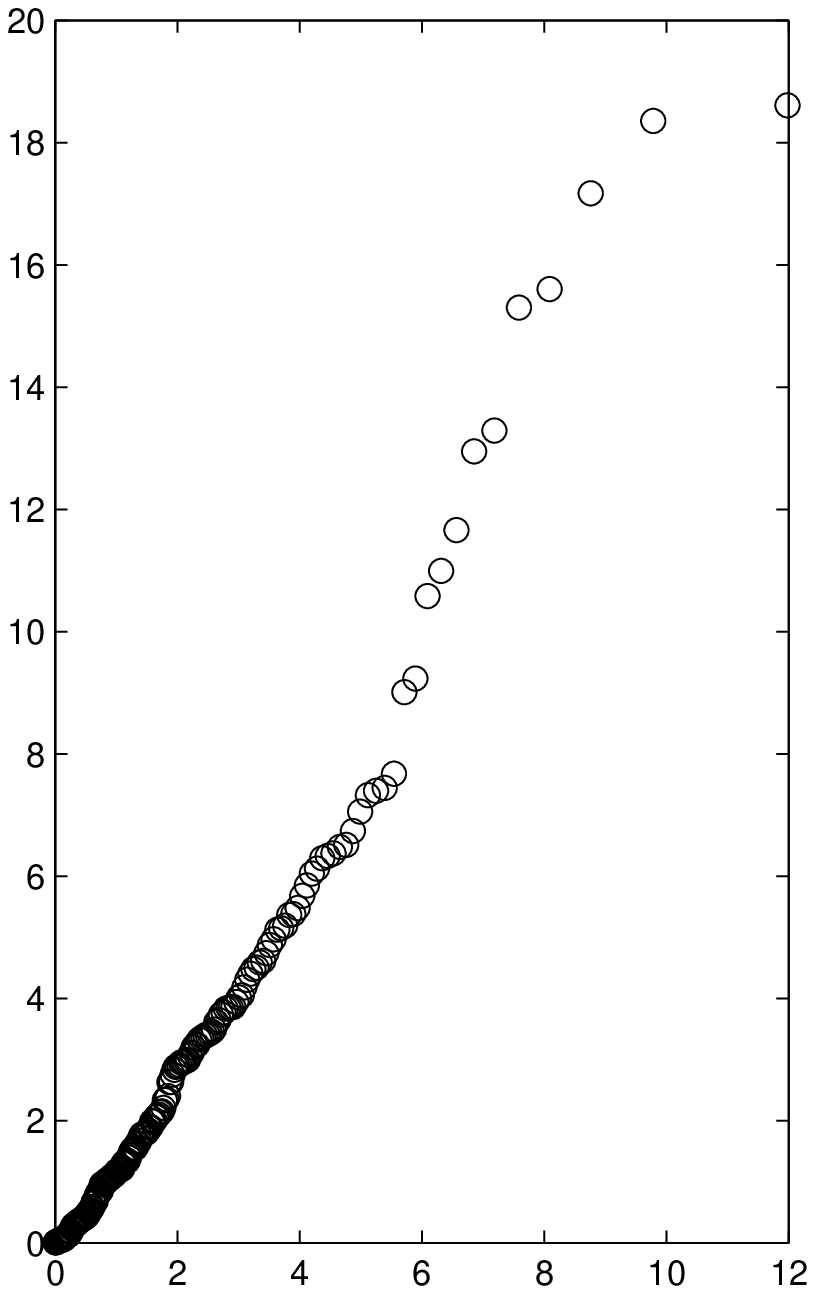,width=0.28\linewidth}}\qquad \qquad
\subfigure[$n=30$]{\epsfig{file=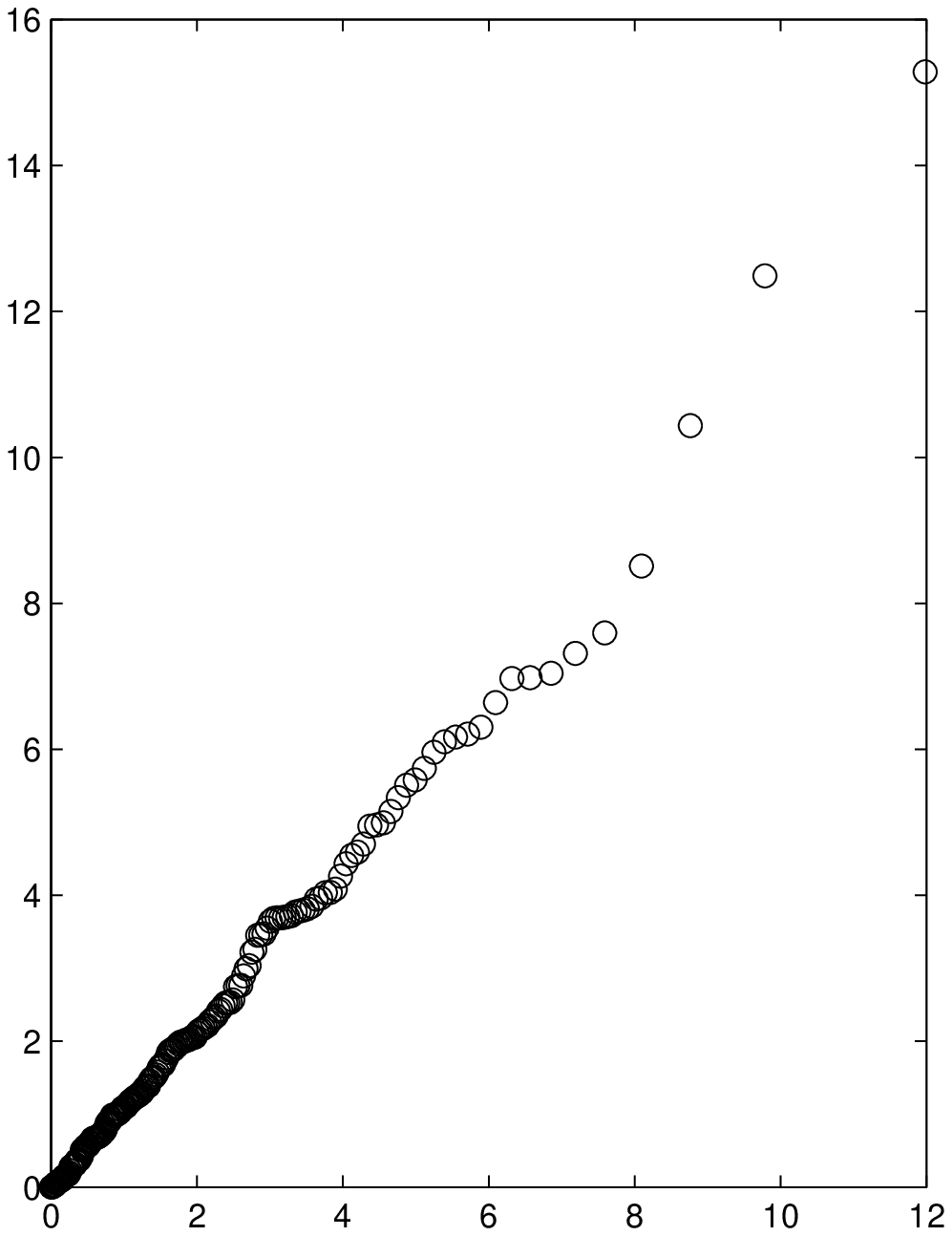,width=0.34\linewidth}}}
\label{f:chi2plot}
\end{figure}

\vspace{4mm}

\noindent \textbf{Examples with $q=10$.} We let $q=10$, $d=110$, $S=\Rset^{10}_+$, and $F:\Rset^{10}\times \Rset^{110} \to \Rset^{10}$ be defined as $F(x,\xi)=\Lambda(\xi) x + b(\xi)$, where $\Lambda(\xi)$ is a $10\times 10$ matrix whose entries are the first 100 components of $\xi$, and $b(\xi)\in \Rset^{10}$ consists of the last 10 components of $\xi$. Each diagonal entry of $\Lambda(\xi)$ is uniformly distributed on the interval [0,4], each entry above the main diagonal is uniformly distributed on [0,3], and each entry below the main diagonal is uniformly distributed on [0,2]. Thus, $E[\Lambda(\xi)]=\Lambda_0$ with $(\Lambda_0)_{ii}=2$, $(\Lambda_0)_{ij}=1.5$ for $i<j$ and $(\Lambda_0)_{ij}=1$ for $i>j$. We consider three different choices for the uniform distribution of $b(\xi)$, to obtain three different examples. In example 1, each component of $b(\xi)$ is uniformly distributed on [-1,1], so $E[b(\xi)]_i=0$ for each $i$. In example 2, the first five components of $b(\xi)$ are uniformly distributed on [-1,0.8] and the last five components uniformly distributed on [-1,1]. In example 3, each component of $b(\xi)$ is uniformly distributed on [-1,0.8]. The true solution $z_0$ is given by
\[
\begin{array}{l}
z_0=0\in \Rset^{10} \text{ in example 1,}\\
z_0= [0.21, \ 0.43, \ 0.85, \ 1.7,  \ 3.4, \   -6.6, \   -6.6, \   -6.6, \   -6.6, \   -6.6]^T \times 10^{-2} \text{ in example 2,}\\
z_0=[0.01,\ 0.01, \ 0.03, \ 0.05, \    0.1, \    0.21, \  0.42, \  0.83, \  1.67, \ 3.34]^T \times 10^{-2}  \text{ in example 3.}
\end{array}
\]
For each example, we generate 200 SAA problems with $n=50$, compute the SAA solutions, and obtain simultaneous and individual confidence intervals for $z_0$ of levels 90\%, 95\% and 99\% from each SAA solution. Table \ref{tab:ind_simu_conf_q10} lists the averages of the $90\%$ confidence intervals for $(z_0)_1$ and $(z_0)_{10}$ obtained from the 200 SAA problems, for each example. Table \ref{tab:example_q10} is analogous to Table \ref{tab:ind_conf_cov}. It summarizes the joint coverage of the true solution by simultaneous confidence intervals obtained from the 200 SAA problems for each example, and the coverage of $(z_0)_1$ and $(z_0)_{10}$ by the corresponding individual confidence intervals.

\begin{table}[ht]
\caption{Average $90\%$ confidence intervals for $(z_0)_1$ and $(z_0)_{10}$ in examples $q=10$ over 200 SAA problems}
\begin{center}
\begin{tabular}{c|ll|ll|ll}
\hline
 & \multicolumn{2}{|c|}{Example 1} & \multicolumn{2}{|c}{Example 2} & \multicolumn{2}{|c}{Example 3}\\
  & Sim CI    &  Ind CI  &  Sim CI    &  Ind CI & Sim CI    &  Ind CI\\
\hline
$(z_0)_1$ &  [-0.49, 0.30] & [-0.26, 0.07] & [-0.40, 0.30] & [-0.19, 0.09]  & [-0.44, 0.28] & [-0.23, 0.07]\\
$(z_0)_{10}$ & [-0.41, 0.28] & [-0.20, 0.08] & [-0.46, 0.25] & [-0.25, 0.04] & [-0.33, 0.30] & [-0.15, 0.11]\\
\hline
\end{tabular}
\end{center}
\label{tab:ind_simu_conf_q10}
\end{table}


\begin{table}[h]
\caption{True solution coverage by confidence intervals in examples with $q=10$ from 200 SAA problems}
\begin{center}
\begin{tabular}{c|ccc|ccc|ccc}
\hline
 & \multicolumn{3}{|c|}{Example 1} & \multicolumn{3}{|c}{Example 2} & \multicolumn{3}{|c}{Example 3}\\
     & $\alpha=$0.1 & 0.05 & 0.01  & $\alpha=$0.1 & 0.05 & 0.01 & $\alpha=$0.1 & 0.05 & 0.01 \\
\hline
 $z_0$ &  198 &  199 &  200 &  196 &  198 &  200 & 197 & 197 & 198\\ \hline
 $(z_0)_1$ &   156 &  178 &  194 &  172 &  186 &  192 & 164 & 179 & 195
 \\
 $(z_0)_{10}$ &    173  & 183 &  195 &    176 &  188 &   197 & 172 & 184 & 195
\\ \hline
\end{tabular}
\end{center}
\label{tab:example_q10}
\end{table}

\Appendix

\vspace{4mm}

\noindent {\bf Acknowledgments.}
 Research of the author was supported by National Science Foundation under the grant DMS-1109099. The author thanks Amarjit Budhiraja, Michael Lamm, Yufeng Liu, Stephen M. Robinson and Liang Yin for helpful discussions related to this research, and the referees and the associate editor for comments that have improved the presentation of this paper. Michael Lamm observed that \eqref{eqn:prob_intersection} holds when $k=2$ and provided the idea for the current proof of Proposition \ref{p:R_0}.

\bibliographystyle{siam}
\bibliography{lu}

\end{document}